\documentclass{amsart}
\usepackage{amssymb,amsmath,amsthm,mathrsfs,multirow,xcolor,framed,url}
\usepackage{float}
\usepackage{amsaddr}
\usepackage{tabularx}
\usepackage{longtable} 
\usepackage{threeparttablex} 
\usepackage{booktabs} 
\usepackage{array}
\usepackage[title]{appendix}
\usepackage{tikz}
\usetikzlibrary{arrows.meta}
\usepackage[
bookmarks=true,
bookmarksnumbered=true,
bookmarksopen=false,
hypertexnames=false,
breaklinks=true,
colorlinks=true,
pagebackref=true,
hyperindex=true,
plainpages=false,
pdfauthor={Yifan Luo, Haigang Zhou},
pdftitle={An explicit type number formula for quaternion orders of level (N1,N2)},
pdfsubject={},
pdfkeywords={ternary quadratic forms, quaternion algebras, local densities, Siegel-Weil formula, type number}
]{hyperref}

\usepackage{tikz-cd}

\newtheorem{theorem}{Theorem}[section]
\newtheorem{lemma}[theorem]{Lemma}
\newtheorem{cor}[theorem]{Corollary}

\newtheorem{prop}[theorem]{Proposition}

\theoremstyle{definition}
\newtheorem{definition}[theorem]{Definition}

\theoremstyle{remark}
\newtheorem{remark}{Remark}
\numberwithin{equation}{section}

\newcommand\nutwid{\overset {\text{\lower 3pt\hbox{$\sim$}}}\nu}

\newcommand{\tr}{\operatorname{tr}}
\newcommand{\n}{\operatorname{n}}
\newcommand{\card}{\operatorname{card}}
\newcommand{\discrd}{\operatorname{discrd}}
\newcommand{\disc}{\operatorname{disc}}
\newcommand{\Aut}{\operatorname{Aut}}
\allowdisplaybreaks

\begin{document}
	
	\title{An explicit type number formula for quaternion orders of level $(N_1,N_2)$}
	
	\author{Yifan Luo\textsuperscript{1} \quad
		Haigang Zhou\textsuperscript{2}}
	
	\address{Institute of Mathematics, Henan Academy of Sciences, Zhengzhou, China, 450046\textsuperscript{1}\\
		School of Mathematical Sciences, Tongji University, Shanghai, China, 200092\textsuperscript{2}}
	\email{kotori@hnas.ac.cn}
	\email{haigangz@tongji.edu.cn}
	
	\begin{abstract}
We give an explicit type number formula for the quaternion orders
of level $(N_1,N_2)$ specified in this paper, where
$N_1=p_1^{2u_1+1}\cdots p_w^{2u_w+1}$, the primes $p_i$ are
distinct, $u_i\ge0$, $w$ is odd, and $\gcd(N_1,N_2)=1$.
The formula includes Eichler orders and the orders with odd
prime-power ramified level considered by Boyd, and allows
nonmaximal local orders at several ramified primes, including $2$.
We express the answer in terms of a generalized modified Hurwitz
class number and explicit local correction factors.
The proof combines the correspondence between quaternion orders
and ternary quadratic forms with the Siegel--Weil formula and
local representation densities.
We tabulate class and type numbers for $N_1N_2\le100$.
As an application, a mass bound and a finite exact calculation
determine the $27$ pairs in this family with type number one.
\end{abstract}
	
	\maketitle

	\section{Introduction}\label{Sec:1}
	
	The type number of quaternion orders—the number of orders up to conjugacy in a definite quaternion algebra—is a fundamental invariant that connects arithmetic, geometry, and the theory of automorphic forms. Its systematic study began with Eichler~\cite{Eic56}, who established a formula for definite hereditary orders when $N_1N_2$ is squarefree, though his result was later corrected by Peters~\cite{Pet69}.
	
	When $N_1$ is squarefree, orders of level $(N_1,N_2)$ are Eichler orders, and the type number formula has been extensively studied. Pizer~\cite{Piz73} gave an explicit formula in this case, while Vignéras~\cite{Vig80} provided a more structural formulation expressed as an unevaluated sum, and Körner~\cite{Kor87} developed a general type number formula. Hasegawa and Hashimoto~\cite{HH95} reinterpreted the type number for Eichler orders in terms of dimensions of certain subspaces of cusp forms of weight $2$ using Atkin--Lehner involutions, a method later generalized by Hashimoto~\cite{Ha06}.
	
	For prime power levels $N_1 = p^{2u+1}$ with $p \neq 2$, Boyd~\cite{B94} derived the type number formula using the Selberg trace formula and optimal embedding theory. More recently, Boylan, Skoruppa, and the second author~\cite{BSZ19} gave a new proof for the type number of maximal orders via modified Hurwitz class numbers. This approach was extended by Li, Skoruppa, and the second author~\cite{LSZ21} to hereditary orders (where $N_1N_2$ is squarefree) using the modified Hurwitz class number $H^{(N_1,N_2)}(D)$.
	
	\medskip
	
	\noindent\textbf{Main result.}
	In this paper, we establish an explicit formula for the type number of orders of level $(N_1, N_2)$ (see Definition~\ref{deffororders}), where
	\[
	N_1 = p_1^{2u_1+1} \cdots p_w^{2u_w+1}
	\]
	with distinct primes $p_1,\ldots,p_w$, nonnegative integers $u_1,\ldots,u_w$, an odd positive integer $w$, and a positive integer $N_2$ such that $\gcd(N_1, N_2) = 1$. Let $N_0=p_1\cdots p_w$ be the product of all primes dividing $N_1$. 
Our aim is a uniform expression in terms of modified Hurwitz
class numbers for the orders in Definition~\ref{deffororders}.
The new local calculations concern higher powers in the ramified
level, their simultaneous occurrence at several primes, and the
dyadic cases. The formula specializes to the Eichler and odd
prime-power settings treated in the earlier literature.
	\begin{theorem}\label{the:type}
		Let $Q_{N_0}$ be the quaternion algebra over $\mathbb Q$ ramified
exactly at the finite primes dividing $N_0$ and at $\infty$. The class number of orders of level $(N_1,N_2)$ in $Q_{N_0}$ is
		\begin{equation*}
			h_{N_1,N_2} = \frac{1}{2} H^{(N_1,N_2)}(4) + H^{(N_1,N_2)}(3) + H^{(N_1,N_2)}(0),
		\end{equation*}
		and the type number is
		\begin{equation*}
			T_{N_1,N_2} = 2^{-e(N_1N_2)-1} \sum_{\substack{n \parallel N_1N_2}} \sum_{\substack{n \mid r \\ r^2 \leq 4n \\ (n,r) \neq (4,\pm 4)}} H^{(N_1,N_2)}(4n - r^2) \prod_{p \mid n} \frac{1 - \left(\frac{\Delta(-4n)}{p}\right)/p}{B_p(n) C_p(n)}.
		\end{equation*}
	\end{theorem}
	
	The notation in Theorem~\ref{the:type} is as follows:
	\begin{itemize}
		\item $e(N_1N_2)$ is the number of prime divisors of $N_1N_2$.
		\item $n \parallel N_1N_2$ means $n \mid N_1N_2$ and $\gcd(n, N_1N_2/n) = 1$ (i.e., $n$ is a unitary divisor).
		\item $H^{(N_1,N_2)}(D)$ is the modified Hurwitz class number (Definition~\ref{def:H}).
		\item $\Delta(-4n)$ denotes the fundamental discriminant of $-4n$.
		\item $\left(\frac{\cdot}{p}\right)$ is the Kronecker symbol.
		\item $v_p(n)$ denotes the exponent of $p$ in the prime factorization of $n$.
		\item For a prime $p$ dividing $n$,
		\begin{equation}\label{eq:Bp}
			B_p(n) =
			\begin{cases}
				(p+1)p^{v_p(n)/2-1} & \text{if } v_p(n) \text{ is even}, \\
				p^{(v_p(n)-1)/2} & \text{if } v_p(n) \text{ is odd},
			\end{cases}
		\end{equation}
		and
		\begin{equation}\label{eq:Cp}
			C_p(n) =
			\begin{cases}
				2 & \text{if } p = 2, \; 4 \mid n, \text{ and } \Delta(-4n) \equiv 5 \pmod{8}, \\
				1 & \text{otherwise}.
			\end{cases}
		\end{equation}
	\end{itemize}
	
	\begin{remark}
		\textup{(i)} When $v_p(n) = 1$ for all $p \mid n$, one verifies that
		\[
		1 - \left(\frac{\Delta(-4n)}{p}\right)/p = B_p(n) = C_p(n) = 1,
		\]
		so the product in Theorem~\ref{the:type} equals $1$. Thus, when $N_1$ and $N_2$ are squarefree (i.e., all orders are hereditary), our formula reduces to that of Li, Skoruppa, and the second author~\cite{LSZ21}.
		
		\noindent\textup{(ii)} When $N_1$ is squarefree and $N_2$ is arbitrary, our result recovers Pizer's formula~\cite{Piz73} for Eichler orders.
		
		\noindent\textup{(iii)} When $N_1 = p^{2u+1}$ for an odd prime $p$, our formula specializes to Boyd's result~\cite{B94}.
		
		\noindent\textup{(iv)} For $n \geq 4$ and $n\mid r$ with $(n,r) \neq (4, \pm 4)$, the condition $r^2 \leq 4n$ forces $r = 0$, simplifying the inner sum.
	\end{remark}
	
The type number has a direct orbit-counting interpretation.
The group of invertible two-sided ideals modulo rational scaling
is an elementary abelian group of order $2^{e(N_1N_2)}$ acting
on locally principal ideal classes; its orbits are the order
types. The counting argument below expresses the stabilizers
through primitive normalizing elements of prescribed reduced
norm and trace. The passage to ternary norm lattices then permits
these weighted counts to be evaluated by the Siegel--Weil formula.
	\medskip
	
	\noindent\textbf{Strategy of proof.}
	The key to our approach is a natural generalization of the modified Hurwitz class number $H^{(N_1,N_2)}(D)$, originally defined in~\cite{LSZ21} only for squarefree levels. We establish a bijection between orders of level $(N_1, N_2)$ and ternary quadratic forms in a suitable genus, which reduces the counting problem to computing representation numbers of forms. The Siegel--Weil formula then expresses these representation numbers in terms of $H^{(N_1,N_2)}(D)$ and local correction factors, which we evaluate explicitly via local density computations at each ramified prime.
	
	\medskip
	
	\noindent\textbf{Classification of type number $1$.}
	As an application, we extend the classification of orders with type number $1$. When $N_1$ is squarefree and $N_2 = 1$, Boylan, Skoruppa, and the second author~\cite{BSZ19} found exactly $9$ such levels:
	\begin{equation*}
		\{(2,1), (3,1), (5,1), (7,1), (13,1), (30,1), (42,1), (70,1), (78,1)\}.
	\end{equation*}
	For hereditary orders (squarefree $N_1 N_2$), Li, Skoruppa, and the second author~\cite{LSZ21} found an additional $11$ pairs:
	\begin{equation*}
		\{(2, 3), (2, 5), (2, 7), (2, 11), (2, 15), (2, 23), (3, 2), (3, 5), (3, 11), (5, 2), (7, 3)\}.
	\end{equation*}
	In this paper, using Theorem~\ref{the:type}, we find $7$ new pairs with non-squarefree levels:
	\begin{equation*}
		\{(2, 9), (3, 4), (3, 8), (5, 4), (8, 1), (8, 5), (168, 1)\}.
	\end{equation*}
	This gives a complete list of $\mathbf{27}$ exceptional levels $(N_1, N_2)$ with type number $1$.
	
The proof first applies the mass formula to bound the prime
divisors and their exponents, and then evaluates the explicit
formula at the resulting finite set of candidates.
The largest level in the type-one list is $168=2^3\cdot3\cdot7$.

This application should be viewed in relation to existing
classifications of single-class genera and orders of small
type number, notably \cite{LK13,KL16}.
Our contribution here is the explicit description in the
$(N_1,N_2)$ parametrization and its derivation from the formula
of Theorem~\ref{the:type}; the seven non-hereditary pairs are
additional to the maximal and hereditary lists just mentioned.

\medskip
\noindent\textbf{Computational results.}
Appendix~\ref{App:A} tabulates the class and type numbers for
$N_1N_2\le100$ and gives further examples in the prime-power
setting studied by Boyd~\cite{B94}.
All entries are obtained by exact evaluation of
Theorem~\ref{the:type}.

	\medskip
	
	\noindent\textbf{Organization of the paper.}
	The remainder of this paper is organized as follows. In Section~\ref{Sec:2} and Section~\ref{Sec:3}, we recall the definition of orders of level $(N_1, N_2)$ and establish the bijection to ternary quadratic forms. Section~\ref{Sec:4} develops the generalized modified Hurwitz class number $H^{(N_1,N_2)}(D)$ and derives its explicit formula via the Siegel--Weil formula. In Section~\ref{Sec:5}, we prove Theorem~\ref{the:type} by combining the bijection with local density calculations. The computational results are deferred to Appendix~\ref{App:A}, and technical details on local densities appear in Appendix~\ref{App:B}.

	\section{Quaternion algebras and orders}\label{Sec:2}
	
	In this section, we establish the foundational framework for quaternion algebras and orders of level $(N_1, N_2)$. Our presentation follows the standard references~\cite{Voi21, Piz76, Piz77}, with particular attention to the structure of normalizers and two-sided ideals, which play a crucial role in the computation of type numbers.
	
	\subsection{Quaternion algebras over local and global fields}
	
	\subsubsection{Basic definitions}
	
	Let $F$ be a field of characteristic $0$ and $a, b \in F^\times$. The \emph{quaternion algebra} $Q = \left(\frac{a,b}{F}\right)$ is the $F$-algebra with basis $\{1, i, j, k\}$ subject to the relations
	\[
	i^2 = a, \quad j^2 = b, \quad k = ij = -ji.
	\]
	Throughout this paper, $F$ denotes either the rational numbers $\mathbb{Q}$, the real numbers $\mathbb{Q}_\infty = \mathbb{R}$, or the $p$-adic numbers $\mathbb{Q}_p$ for a prime $p$.
	
	Over a local field $\mathbb{Q}_p$ (or $\mathbb{Q}_\infty$), there are exactly two isomorphism classes of quaternion algebras: the matrix algebra $M_2(\mathbb{Q}_p)$ and the unique division algebra (a skew-field). The standard $p$-adic Hilbert symbol $(a, b)_p$ takes values in $\{\pm 1\}$ and determines the isomorphism class:
	\[
	Q_p = \left(\frac{a,b}{\mathbb{Q}_p}\right) \simeq M_2(\mathbb{Q}_p) \quad \Longleftrightarrow \quad (a, b)_p = 1.
	\]
	At the archimedean place, $(a, b)_\infty = -1$ if and only if $a, b < 0$.
	
	\subsubsection{Global quaternion algebras}
	
	For quaternion algebras over $\mathbb{Q}$, two algebras $Q = \left(\frac{a,b}{\mathbb{Q}}\right)$ and $Q' = \left(\frac{a',b'}{\mathbb{Q}}\right)$ are isomorphic if and only if
	\[
	(a, b)_p = (a', b')_p \quad \text{for all primes } p \text{ and } \infty.
	\]
	A quaternion algebra $Q$ over $\mathbb{Q}$ is said to \emph{ramify} at a place $v$ (finite prime $p$ or $\infty$) if $\mathbb{Q}_v \otimes_\mathbb{Q} Q$ is a division algebra, and to \emph{split} at $v$ if $\mathbb{Q}_v \otimes_\mathbb{Q} Q \simeq M_2(\mathbb{Q}_v)$. A quaternion algebra ramifies at finitely many places, and by the Hasse--Minkowski theorem, the number of ramified places is always even. A quaternion algebra is called \emph{definite} if it ramifies at the infinite place $\infty$.
	
	Given a finite set $S$ of an even number of primes, there exist $a, b \in \mathbb{Q}^\times$ such that $(a, b)_p = -1$ if and only if $p \in S$. The resulting algebra $Q = \left(\frac{a,b}{\mathbb{Q}}\right)$ is (up to isomorphism) the unique quaternion algebra ramifying exactly at the primes in $S$. The \emph{reduced discriminant} $\disc(Q)$ is defined as the product of all finite primes at which $Q$ ramifies.
	
	\subsubsection{Reduced trace and norm}
	
	For any element $\alpha = t + xi + yj + zk \in Q$, the \emph{conjugation} is defined by
	\[
	\overline{\alpha} = t - xi - yj - zk.
	\]
	This involution satisfies $\overline{\overline{\alpha}} = \alpha$ and $\overline{\alpha\beta} = \overline{\beta} \, \overline{\alpha}$ for all $\alpha, \beta \in Q$. The \emph{reduced trace} and \emph{reduced norm} are given by
	\[
	\tr(\alpha) = \alpha + \overline{\alpha}, \quad \n(\alpha) = \alpha \cdot \overline{\alpha}.
	\]
	When $Q \simeq M_2(F)$ and
	\[
	A = \begin{pmatrix} a & b \\ c & d \end{pmatrix} \in Q,
	\]
	we have
	\[
	\overline{A} = \begin{pmatrix} d & -b \\ -c & a \end{pmatrix}, \quad \tr(A) = a + d, \quad \n(A) = ad - bc = \det(A).
	\]
	Every element $\alpha \in Q$ satisfies its \emph{reduced characteristic polynomial}
	\[
	x^2 - \tr(\alpha) x + \n(\alpha) = 0.
	\]
	
	\subsection{Orders and ideals}
	
	\subsubsection{Lattices and orders}
	
	An \emph{ideal} in a quaternion algebra $Q$ over $\mathbb{Q}$ is a full $\mathbb{Z}$-lattice, i.e., a finitely generated $\mathbb{Z}$-submodule containing a $\mathbb{Q}$-basis of $Q$. An \emph{order} $\mathcal{O}$ in $Q$ is an ideal that is also a ring with $\mathbb{Z} \subseteq \mathcal{O}$. For the remainder of this paper, $\mathcal{O}$ always denotes an order in a definite quaternion algebra.
	
	If $x \in \mathcal{O}$, then $\tr(x), \n(x) \in \mathbb{Z}$. An order $\mathcal{O}$ is \emph{maximal} if it is not properly contained in any other order. Over a local field $\mathbb{Q}_p$, a quaternion division algebra contains a unique maximal order (up to conjugation).
	
	\subsubsection{Local structure of maximal orders}
	
	For an odd prime $p$, let $\epsilon \in \mathbb{Z}_p^\times$ with $\left(\frac{\epsilon}{p}\right) = -1$, and set $i^2 = \epsilon$, $j^2 = p$. The unique maximal order in the quaternion division algebra $Q_p$ over $\mathbb{Q}_p$ is
	\[
	\mathcal{O}_p = R_p + R_p j, \quad \text{where } R_p = \mathbb{Z}_p + i\mathbb{Z}_p.
	\]
	For $p=2$, take $i^2=-3$ and $j^2=2$, with $ij=-ji$, and put
\[
R_2=\mathbb Z_2\left[\frac{1+i}{2}\right].
\]
This is the ring of integers of the unramified quadratic
extension of $\mathbb Q_2$, and
$\mathcal O_2=R_2+R_2j$ is the maximal order in the division
quaternion algebra over $\mathbb Q_2$. In matrix form,
	\[
	Q_p \simeq \left\{ \begin{pmatrix} \alpha & \beta \\ p\overline{\beta} & \overline{\alpha} \end{pmatrix} : \alpha, \beta \in \mathbb{Q}_p + i\mathbb{Q}_p \right\},
	\]
	and the maximal order is
	\[
	\mathcal{O}_p \simeq \left\{ \begin{pmatrix} \alpha & \beta \\ p\overline{\beta} & \overline{\alpha} \end{pmatrix} : \alpha, \beta \in R_p \right\}.
	\]
	For a nonnegative integer $u$, define
	\[
	\mathcal{O}_p^{(2u+1)} = \left\{ \begin{pmatrix} \alpha & p^u \beta \\ p^{u+1} \overline{\beta} & \overline{\alpha} \end{pmatrix} : \alpha, \beta \in R_p \right\}.
	\]
	One verifies that $\mathcal{O}_p^{(2u+1)}$ is an order in $Q_p$ with index $[\mathcal{O}_p : \mathcal{O}_p^{(2u+1)}] = p^{2u}$.
	
	\subsection{Orders of level \texorpdfstring{$(N_1, N_2)$}{(N1, N2)}}
	
	Throughout this paper, we assume that $Q$ is a \emph{definite} quaternion algebra over $\mathbb{Q}$. By the ramification condition, $Q$ must ramify at an odd number of finite primes (and at the infinite place).
	
	\begin{definition}[\cite{Piz76}]\label{deffororders}
		Let $N_1 = p_1^{2u_1+1} \cdots p_w^{2u_w+1}$, where $p_1, \ldots, p_w$ are distinct primes, $u_1, \ldots, u_w$ are nonnegative integers, and $w$ is an odd integer. Let $N_2$ be a positive integer with $\gcd(N_1, N_2) = 1$. An order $\mathcal{O}$ in a definite quaternion algebra $Q_{N_0}$ ramifying only at $\{p_1, \ldots, p_w, \infty\}$ is said to be \emph{of level $(N_1, N_2)$} if its local completions satisfy:
		\begin{enumerate}
			\item For $1 \leq i \leq w$, $\mathcal{O}_{p_i}$ is isomorphic (over $\mathbb{Z}_{p_i}$) to $\mathcal{O}_{p_i}^{(2u_i+1)}$.
			\item For primes $p \nmid N_1$, $\mathcal{O}_p$ is isomorphic (over $\mathbb{Z}_p$) to
			\[
			\begin{pmatrix} \mathbb{Z}_p & \mathbb{Z}_p \\ N_2 \mathbb{Z}_p & \mathbb{Z}_p \end{pmatrix}.
			\]
		\end{enumerate}
	\end{definition}
	
	\begin{remark}
		\textup{(i)} If $N_1$ is squarefree and $N_2 = 1$, orders of level $(N_1, 1)$ are maximal orders.
		
		\noindent\textup{(ii)} If $N_1 N_2$ is squarefree, orders of level $(N_1, N_2)$ are hereditary orders (also called Eichler orders of squarefree level).
		
		\noindent\textup{(iii)} If $N_1$ is squarefree (but $N_2$ is arbitrary), orders of level $(N_1, N_2)$ are Eichler orders.
	\end{remark}
	
	\subsection{Left and right ideals}
	
	\begin{definition}
		Let $\mathcal{O}$ be an order of level $(N_1, N_2)$. A \emph{left $\mathcal{O}$-ideal} is an ideal $I$ in $Q$ such that $I_p = \mathcal{O}_p x_p$ for every prime $p$, where $x_p \in (\mathbb{Q}_p \otimes_\mathbb{Q} Q)^\times$. The notion of a \emph{right $\mathcal{O}$-ideal} is defined analogously.
	\end{definition}
	
	Two orders $\mathcal{O}$ and $\mathcal{O}'$ are said to be \emph{of the same type} (or \emph{isomorphic as $\mathbb{Z}$-algebras}) if there exists $\alpha \in Q^\times$ such that $\mathcal{O}' = \alpha^{-1} \mathcal{O} \alpha$. Orders $\mathcal{O}$ and $\mathcal{O}'$ are \emph{locally isomorphic} (or in the same \emph{genus}) if $\mathcal{O}_p$ and $\mathcal{O}'_p$ are conjugate in $Q_p$ for all primes $p$.
	
	The \emph{type set} of $\mathcal{O}$ is the set of isomorphism classes of orders in the genus of $\mathcal{O}$, and its cardinality is the \emph{type number} $T_{N_1, N_2}$. The \emph{class number} $h_{N_1, N_2}$ is the number of equivalence classes of left $\mathcal{O}$-ideals modulo multiplication by $Q^\times$ on the right. Both $T_{N_1, N_2}$ and $h_{N_1, N_2}$ are finite and depend only on $(N_1, N_2)$.
	
	\subsection{Two-sided ideals and normalizers}
	
	To compute type and class numbers, we study two-sided ideals and their relation to normalizers.
	
	\begin{definition}
		An ideal $I$ is a \emph{two-sided $\mathcal{O}$-ideal} if it is both a left and a right $\mathcal{O}$-ideal.
	\end{definition}
	
	An ideal $I$ is two-sided if and only if there exist $x_p, y_p \in (\mathbb{Q}_p \otimes_\mathbb{Q} Q)^\times$ such that $I_p = \mathcal{O}_p x_p = y_p \mathcal{O}_p$ for every prime $p$. This is equivalent to $y_p = u_p x_p$ for some $u_p \in \mathcal{O}_p^\times$, which yields $\mathcal{O}_p = x_p^{-1} \mathcal{O}_p x_p$. This motivates the definition of the \emph{normalizer} of $\mathcal{O}_p$ in $Q_p$:
	\[
	N(\mathcal{O}_p) = \{ x_p \in (\mathbb{Q}_p \otimes_\mathbb{Q} Q)^\times : x_p^{-1} \mathcal{O}_p x_p = \mathcal{O}_p \}.
	\]
	
	\begin{prop}\label{propfortwosidedideal}
		An ideal $I$ is a two-sided $\mathcal{O}$-ideal if and only if there exists $x_p \in N(\mathcal{O}_p)$ such that $I_p = \mathcal{O}_p x_p$ for every prime $p$.
	\end{prop}
	
	The structure of $N(\mathcal{O}_p)$ is given by the following result.
	
	\begin{lemma}[\cite{Piz77}]\label{lemmafornormalizer}
		Let $\mathcal{O}$ be an order of level $(N_1, N_2)$. Then:
		\begin{itemize}
			\item If $p \nmid N_1 N_2$, then $N(\mathcal{O}_p) = \mathcal{O}_p^\times \mathbb{Q}_p^\times$.
			\item If $p^{2u+1} \parallel N_1$, then
			\[
			N(\mathcal{O}_p) = \mathcal{O}_p^\times \mathbb{Q}_p^\times \cup \begin{pmatrix} 0 & p^u \\ p^{u+1} & 0 \end{pmatrix} \mathcal{O}_p^\times \mathbb{Q}_p^\times.
			\]
			\item If $p^v \parallel N_2$, then
			\[
			N(\mathcal{O}_p) = \mathcal{O}_p^\times \mathbb{Q}_p^\times \cup \begin{pmatrix} 0 & 1 \\ p^v & 0 \end{pmatrix} \mathcal{O}_p^\times \mathbb{Q}_p^\times.
			\]
		\end{itemize}
	\end{lemma}
	
	\subsection{Two-sided principal ideals}
	
	We now characterize elements $x \in \mathcal{O}$ such that $\mathcal{O} x$ is a two-sided principal ideal.
	An element $0\ne x\in\mathcal O$ is called \emph{primitive} if
$x\notin p\mathcal O$ for every prime $p$.
Write
\[
N_Q(\mathcal O)=\{x\in Q^\times:x^{-1}\mathcal O x=\mathcal O\}.
\]

\begin{lemma}\label{lemmaforprincipalideal}
Let $N=N_1N_2$, and let $x\in\mathcal O$ be nonzero and primitive.
Then $\mathcal O x$ is a two-sided $\mathcal O$-ideal if and only if
\begin{enumerate}
\item $\n(x)\parallel N$;
\item $\n(x)\mid\tr(x)$;
\item for every $p^g\parallel\n(x)$, one has
$c_{11}(x_p)\in p^gR_p$ when $p\mid N_1$, and
$c_{11}(x_p)\in p^g\mathbb Z_p$ when $p\mid N_2$.
\end{enumerate}
Here the matrix entries are taken in the local models of
Definition~\ref{deffororders}.
\end{lemma}

\begin{proof}
Suppose first that $x\in N_Q(\mathcal O)$.
By Lemma~\ref{lemmafornormalizer}, after absorbing a scalar unit
into $\mathcal O_p^\times$, the local element $x_p$ has the form
$p^a u_p$ or $p^a w_pu_p$, where $u_p\in\mathcal O_p^\times$ and
\[
w_p=
\begin{cases}
\begin{pmatrix}0&p^u\\p^{u+1}&0\end{pmatrix},
   &p^{2u+1}\parallel N_1,\\[4pt]
\begin{pmatrix}0&1\\p^v&0\end{pmatrix},
   &p^v\parallel N_2.
\end{cases}
\]
Both $u_p$ and $w_pu_p$ are primitive in $\mathcal O_p$.
Indeed, in the division case the diagonal coefficient of $u_p$
is a unit of $R_p$; in the split case both diagonal coefficients
of $u_p$ are units of $\mathbb Z_p$.
The displayed products therefore have a unit coefficient in an
$\mathcal O_p$-basis. Integrality forces $a\ge0$, and primitivity
forces $a=0$. The same argument gives $x_p\in\mathcal O_p^\times$
at primes not dividing $N$.
Thus $v_p(\n(x))$ is either $0$ or $v_p(N)$.
Multiplication by the displayed $w_p$ shows that in the second
case both the trace and the $(1,1)$-entry $c_{11}(x_p)$ are divisible by $p^{v_p(N)}$.
This proves the three conditions.

Conversely, put $n=\n(x)$. If $p\nmid n$, then
$x_p\in\mathcal O_p^\times$. Suppose $p^g\parallel n$.
If $p^g=p^{2u+1}\parallel N_1$, write
\[
x_p=\begin{pmatrix}\alpha&p^u\beta\\
              p^{u+1}\overline\beta&\overline\alpha\end{pmatrix}.
\]
The condition $\alpha\in p^{2u+1}R_p$ implies
\[
w_p^{-1}x_p=
\begin{pmatrix}
\overline\beta&p^{-u-1}\overline\alpha\\
p^{-u}\alpha&\beta
\end{pmatrix}\in\mathcal O_p.
\]
Its reduced norm is a unit, so it belongs to $\mathcal O_p^\times$.
If $p^g\parallel N_2$, write
$x_p=\left(\begin{smallmatrix}a&b\\p^gc&d\end{smallmatrix}\right)$.
Conditions (2) and (3) give $p^g\mid a,d$, and hence
\[
w_p^{-1}x_p=
\begin{pmatrix}c&p^{-g}d\\a&b\end{pmatrix}\in\mathcal O_p^\times.
\]
Thus $x_p$ normalizes $\mathcal O_p$ at every prime.
\end{proof}

	\begin{remark}\label{remarkfortwosidedprincipal}
		When $u = 0$ (for primes dividing $N_1$) or $v = 1$ (for primes dividing $N_2$), condition~(3) is automatically satisfied whenever conditions~(1) and~(2) hold. Specifically:
		\begin{itemize}
			\item If $p \mid \n(\alpha)$ and $p \mid \tr(\alpha)$ for $\alpha \in R_p$, then $\alpha\in pR_p$.
			\item If $p \mid ad$ and $p \mid (a + d)$ for $a, d \in \mathbb{Z}_p$, then $p \mid a$.
		\end{itemize}
		Hence, in the hereditary case ($N_1 N_2$ squarefree), condition~(3) can be omitted.
	\end{remark}
	
	\subsection{Counting two-sided principal ideals}
	
	Fix an order $\mathcal{O}$ of level $(N_1, N_2)$. Let $\mathfrak{I}(\mathcal{O})$ denote the group of two-sided $\mathcal{O}$-ideals, and $\mathfrak{B}(\mathcal{O})$ the subgroup of two-sided principal $\mathcal{O}$-ideals. The quotient $\mathfrak{I}(\mathcal{O}) / \mathbb{Q}^\times$ has order $2^{e(N_1 N_2)}$, where $e(N_1 N_2)$ denotes the number of prime divisors of $N_1 N_2$.
	
\begin{theorem}\label{theoremforxingshu}
Let $N=N_1N_2$, let $e=e(N)$, and let $m(\mathcal O)$ be the number
of left $\mathcal O$-ideal classes represented by two-sided ideals.
Define
\[
\rho_{\mathcal O}(n,r)=
\#\{x\in\mathcal O:\n(x)=n,\ \tr(x)=r\},
\]
and
\[
\rho'_{\mathcal O}(n,r)=
\#\{x\in\mathcal O:\ x\text{ is primitive},\
x\in N_Q(\mathcal O),\ \n(x)=n,\ \tr(x)=r\}.
\]
Then
\begin{align}
|\mathfrak B(\mathcal O)/\mathbb Q^\times|
  &=\frac{2^e}{m(\mathcal O)},\label{XINGSHU0}\\
|\mathfrak B(\mathcal O)/\mathbb Q^\times|
  &=\sum_{n\parallel N}\ \sum_{\substack{n\mid r\\r^2\le4n}}
       \frac{\rho'_{\mathcal O}(n,r)}{|\mathcal O^\times|},
       \label{XINGSHU1}\\
|\Aut(\mathcal O)|
  &=\frac{2^e|\mathcal O^\times|}{2m(\mathcal O)}.
       \label{XINGSHU2}
\end{align}
For $n\parallel N$ and $n\mid r$, the definition of $\rho'$ is
equivalent to imposing condition (3) of
Lemma~\ref{lemmaforprincipalideal} on the elements counted by $\rho$.
\end{theorem}
\begin{proof}
The local normalizer description identifies
$\mathfrak I(\mathcal O)/\mathbb Q^\times$ with an elementary
abelian group of order $2^e$. Its action on ideal classes has the
orbit of $[\mathcal O]$ of size $m(\mathcal O)$ and stabilizer
$\mathfrak B(\mathcal O)/\mathbb Q^\times$, proving
\eqref{XINGSHU0}. The primitive-generator description above and
Lemma~\ref{lemmaforprincipalideal} give \eqref{XINGSHU1}.
Skolem--Noether identifies
$\Aut(\mathcal O)=N_Q(\mathcal O)/\mathbb Q^\times$.
The map $x\mapsto\mathcal O x$ induces the exact sequence
\[
1\longrightarrow\mathcal O^\times/\{\pm1\}
\longrightarrow\Aut(\mathcal O)
\longrightarrow\mathfrak B(\mathcal O)/\mathbb Q^\times
\longrightarrow1,
\]
which proves \eqref{XINGSHU2}.
For the last assertion, the converse local calculations in
Lemma~\ref{lemmaforprincipalideal} show directly that the elements
in question are primitive: they are local units or belong to
$w_p\mathcal O_p^\times$.
\end{proof}

Every class in $\mathfrak B(\mathcal O)/\mathbb Q^\times$ has an
integral primitive generator: clear the denominators of a generator
and divide its integral coordinates by their greatest common divisor.
Two such generators representing the same class differ by a unit of
$\mathcal O$. Indeed, a rational scalar carrying one primitive element
to another must be $\pm1$.
Consequently, each class contributes exactly $|\mathcal O^\times|$
primitive generators, all with the same reduced norm.

The map from locally principal ideal classes to types, obtained by
taking the right order, has fibers of size $m(\mathcal O_\mu)$.
This follows by fixing one ideal in a fiber and using the transitive
action of the two-sided ideal classes of its right order.
In particular,
\[
h_{N_1,N_2}=\sum_{\mu=1}^{T_{N_1,N_2}}m(\mathcal O_\mu).
\]

	The following classical result relates the class number to the unit groups of orders.
	
	\begin{prop}[\cite{Piz76}]\label{mass}
		For orders of level $(N_1, N_2)$, the mass formula holds:
		\begin{equation*}
			\sum_{i=1}^{h_{N_1,N_2}} \frac{1}{\card(\mathcal{O}_i^\times)} = \frac{N_1 N_2}{24} \prod_{p \mid N_1} \left(1 - \frac{1}{p}\right) \prod_{p \mid N_2} \left(1 + \frac{1}{p}\right).
		\end{equation*}
	\end{prop}
	
	Combining Theorem~\ref{theoremforxingshu} with the mass formula, we obtain:
	\begin{equation}\label{eq:masstype}
		\sum_{\mu=1}^{T_{N_1,N_2}} \frac{1}{\card(\Aut(\mathcal{O}_\mu))} = 2^{-e(N_1 N_2)} \frac{N_1 N_2}{12} \prod_{p \mid N_1} \left(1 - \frac{1}{p}\right) \prod_{p \mid N_2} \left(1 + \frac{1}{p}\right).
	\end{equation}
	
	By equations~\eqref{XINGSHU0},~\eqref{XINGSHU1}, and~\eqref{XINGSHU2}, we derive the key formula:
	
	\begin{align}\label{eq:typenumber_intermediate}
		T_{N_1,N_2}
		&= 2^{-e(N_1 N_2)} \sum_{\mu=1}^{T_{N_1,N_2}} 2^{e(N_1 N_2)} \notag \\
		&= 2^{-e(N_1 N_2)} \sum_{\mu=1}^{T_{N_1,N_2}} m(\mathcal{O}_\mu) \card(\mathfrak{B}(\mathcal{O}_\mu) / \mathbb{Q}^\times) \notag \\
		&= 2^{-e(N_1 N_2)} \sum_{\mu=1}^{T_{N_1,N_2}} m(\mathcal{O}_\mu) \sum_{n \parallel N_1 N_2} \sum_{\substack{n \mid r \\ r^2 \leq 4n}}\frac{\rho'_{\mathcal{O}_\mu}(n, r)}{\card(\mathcal{O}_\mu^\times)} \notag \\
		&= 2^{-1} \sum_{n \parallel N_1 N_2} \sum_{\substack{n \mid r \\ r^2 \leq 4n}}\sum_{\mu=1}^{T_{N_1,N_2}} \frac{\rho'_{\mathcal{O}_\mu}(n, r)}{\card(\Aut(\mathcal{O}_\mu))}.
	\end{align}
	
	The key step in proving Theorem~\ref{the:type} is to evaluate the weighted sum
	\[
	\sum_{\mu=1}^{T_{N_1,N_2}} \frac{\rho'_{\mathcal{O}_\mu}(n, r)}{\card(\Aut(\mathcal{O}_\mu))}
	\]
	in terms of the modified Hurwitz class number $H^{(N_1, N_2)}(4n - r^2)$. This is accomplished in Section~\ref{Sec:3} via a bijection between orders and ternary quadratic forms, combined with the Siegel--Weil formula.

	\section{Ternary quadratic forms and quaternion orders}\label{Sec:3}
	In this section, let $N_1=p_1^{2u_1+1}\cdots p_w^{2u_w+1}$ and $N_0=p_1\cdots p_w$, where $p_1,\ldots,p_w$ are distinct primes, $u_1,\ldots,u_w$ are nonnegative integers and $w$ is an odd positive integer, and let $N_2$ be a positive integer such that $\gcd(N_1,N_2)=1$.
	Bijections between quaternion orders and ternary quadratic forms allow us to express the counting problems $\rho_\mathcal{O}(n,r)$ and related weighted sums in terms of representation numbers $R_f(D)$, which are computable via the Siegel--Weil formula.
	
	\subsection{\texorpdfstring{Bijections between orders of level $(N_1,N_2)$ and ternary quadratic forms}{Bijections between orders of level (N1,N2) and ternary quadratic forms}}
	In this subsection, we will introduce bijections between orders of level $(N_1,N_2)$ and ternary quadratic forms for given $N_1$ and $N_2$. For more detail and notations about ternary quadratic forms (level $N_f$, discriminant $d_f$, correspondence $\phi_p$, equivalence class and genus $G$), we refer to~\cite[Section 2]{Leh92} and~\cite[Section 2]{LZ26}.
	
	To construct the bijection, we first need to normalize the basis of each order.
	The following proposition guarantees the existence of a basis with prescribed trace
	conditions, which will serve as the foundation for associating quadratic forms to orders.
	This proposition is a straightforward generalization of~\cite[Proposition 4.1]{LZ26}, and its proof follows similarly.
	\begin{prop}\label{propforbase}
		Let $\mathcal{O}=\mathbb{Z}+\mathbb{Z}\alpha_1+\mathbb{Z}\alpha_2+\mathbb{Z}\alpha_3\subseteq Q$
		be an order of level $(N_1,N_2)$. Then there exist
		$\alpha_1',\alpha_2',\alpha_3'\in\mathcal{O}$ with
		$\tr(\alpha_1')=\tr(\alpha_2')=0$ and $\tr(\alpha_3')=1$ such that
		\begin{equation*}
			\mathcal{O}=\mathbb{Z}+\mathbb{Z}\alpha_1'+\mathbb{Z}\alpha_2'+\mathbb{Z}\alpha_3'.
		\end{equation*}
	\end{prop}
	
	Unless otherwise stated, we assume that for an order of level $(N_1,N_2)$, we have
	$\mathcal{O}=\mathbb{Z}+\mathbb{Z}\alpha_1+\mathbb{Z}\alpha_2+\mathbb{Z}\alpha_3$
	with $\tr(\alpha_1)=\tr(\alpha_2)=0$ and $\tr(\alpha_3)=1$.

The trace-zero lattice is
\[
\mathcal O^0=\mathcal O\cap Q^0
 =\mathbb Z\alpha_1+\mathbb Z\alpha_2+
   \mathbb Z(2\alpha_3-1).
\]
Indeed, an element $t+x\alpha_1+y\alpha_2+z\alpha_3$ has trace
$2t+z$, so trace zero is equivalent to $z=-2t$.
	The trace-zero lattice is an even integral lattice when equipped
	with the bilinear form $(x,y)\mapsto \tr(x\overline{y})$. Since the quaternion algebra
	$Q$ is definite, this form is positive definite.
	
	To associate a ternary quadratic form to $\mathcal{O}$, we consider a related lattice.
	Define
	\begin{equation*}
		S=\mathbb{Z}+2\mathcal{O}=\mathbb{Z}+\mathbb{Z}2\alpha_1+\mathbb{Z}2\alpha_2+\mathbb{Z}2\alpha_3,
	\end{equation*}
	and let
	\begin{equation*}
		S^{0}=S\cap Q^0=\mathbb{Z}2\alpha_1+\mathbb{Z}2\alpha_2+\mathbb{Z}(2\alpha_3-1).
	\end{equation*}
	The lattice $S^{0}$, equipped with the bilinear form $(x,y)\mapsto \tr(x\overline{y})$,
	is an even integral lattice, and is also positive definite.
	
	The norm form on $S^{0}$ defines a ternary quadratic form: for $(x,y,z)\in\mathbb{Z}^3$,
	\begin{align*}
		f_{S^{0}}(x,y,z) & =\n(2x\alpha_1+2y\alpha_2+z(2\alpha_3-1))\\
		& =4\n(\alpha_1)x^2+4\n(\alpha_2)y^2+(4\n(\alpha_3)-1)z^2\\
		&\quad +4\tr(\alpha_2\overline{\alpha_3})yz+4\tr(\alpha_1\overline{\alpha_3})xz+4\tr(\alpha_1\overline{\alpha_2})xy,
	\end{align*}
	and the associated matrix of $f_{S^{0}}$ is
	\begin{equation*}
		M_{f_{S^{0}}}=
		\begin{pmatrix}
			4\tr(\alpha_1\overline{\alpha_1}) &4\tr(\alpha_1\overline{\alpha_2}) & 4\tr(\alpha_1\overline{\alpha_3})\\
			4\tr(\alpha_2\overline{\alpha_1}) &4\tr(\alpha_2\overline{\alpha_2}) & 4\tr(\alpha_2\overline{\alpha_3})\\
			4\tr(\alpha_3\overline{\alpha_1}) &4\tr(\alpha_3\overline{\alpha_2}) & 4\tr(\alpha_3\overline{\alpha_3})-2
		\end{pmatrix}.
	\end{equation*}
	Note that $2\n(\alpha)=\tr(\alpha\overline{\alpha})$.
	
	This construction gives a well-defined map from isomorphism classes of orders to
	equivalence classes of ternary quadratic forms:
	
	\begin{prop}\label{propformapswelldefined}
		Let $\mathcal{O},\mathcal{O}'$ be orders in $Q$, with
		associated ternary quadratic forms $f_{S^{0}}$ and $f_{S'^{0}}$.
		\begin{enumerate}
			\item[(i)] If $\mathcal{O}$ and $\mathcal{O}'$ are isomorphic, then $f_{S^{0}}$
			and $f_{S'^{0}}$ are equivalent.
			\item[(ii)] If $\mathcal{O}$ and $\mathcal{O}'$ are locally isomorphic (at all
			primes), then $f_{S^{0}}$ and $f_{S'^{0}}$ are in the same genus.
		\end{enumerate}
	\end{prop}
	\begin{proof}
		The presented proposition extends the result of~\cite[Proposition 4.2]{LZ26} in a natural way; the proof employs essentially the same method.
	\end{proof}
	
	To construct the inverse map from ternary quadratic forms to orders, we require the theory of Clifford algebras. For background on the Clifford algebra, we refer to~\cite[Chapter 22]{Voi21}.
	
	Let $f$ be a non-degenerate integral ternary quadratic form. We define $C_0(f)$ to be the even Clifford algebra over $\mathbb{Z}$ associated with $f$. Then $C_0(f)$ is an order
	in a quaternion algebra $Q$ over $\mathbb{Q}$. Conversely, to an order $\mathcal{O}$ we can associate a ternary quadratic form as
	follows. Let $\mathcal{O}^{\sharp}$ denote the dual lattice of $\mathcal{O}$ with
	respect to the trace pairing, and consider the trace-zero lattice
	\begin{equation*}
		\Lambda =\mathcal{O}^{\sharp}\cap Q^0,
	\end{equation*}
	which is a 3-dimensional $\mathbb{Z}$-lattice in $Q^0$. If
	$\mathcal{O}^{\sharp}= \langle e_0', e_1', e_2', e_3'\rangle$, where
	$\{e_0', e_1', e_2', e_3'\}$ is the dual basis of $\{1, e_1, e_2, e_3\}$ with
	respect to the trace pairing, then $\Lambda=\langle e_1', e_2', e_3'\rangle$.
	We define the ternary quadratic form $f_{\mathcal{O}}$ associated to $\mathcal{O}$ by
	\begin{equation*}
		f_{\mathcal{O}}(x,y,z)=\discrd(\mathcal{O})\cdot \n(xe_1'+ye_2'+ze_3').
	\end{equation*}
	Here $\discrd(\mathcal O)$ denotes the positive generator of the
reduced discriminant ideal of $\mathcal O$.
	
	These two constructions establish a correspondence between orders in quaternion
	algebras and ternary quadratic forms, which we make precise in the following theorem.
	
\begin{theorem}[{\cite[Main Theorem 22.1.1]{Voi21}}]\label{Lem}
Let $R$ be a principal ideal domain. The even Clifford construction
induces a bijection between similarity classes of nondegenerate
integral ternary quadratic forms over $R$ and isomorphism classes
of quaternion $R$-orders. Similarity allows a change of variables
in $\mathrm{GL}_3(R)$ and multiplication of the form by a unit
of $R$. The correspondence preserves discriminant ideals.
For $R=\mathbb Z$, restricting to positive definite forms gives
a bijection between their ordinary integral equivalence classes
and isomorphism classes of orders in definite quaternion algebras.
The inverse sends $\mathcal O$ to the norm on
$\mathcal O^\sharp\cap Q^0$, multiplied by the positive generator
of $\discrd(\mathcal O)$.
\end{theorem}

	\begin{remark}
		In what follows, we specialize to $R = \mathbb{Z}$ and provide explicit formulas for
		the correspondence described in Theorem~\ref{Lem}.
	\end{remark}
	
	Let $f=(a,b,c,r,s,t)$ be a ternary quadratic form with discriminant $d_f=d$.
	The corresponding quaternion order $C_0(f) = \mathbb{Z} + \mathbb{Z}e_1 + \mathbb{Z}e_2 + \mathbb{Z}e_3$
	satisfies the multiplication relations
	\begin{equation*}
		\begin{aligned}
			e_1^2&=re_1-bc, & e_2e_3&=a\overline{e_1},\\
			e_2^2&=se_2-ac, & e_3e_1&=b\overline{e_2},\\
			e_3^2&=te_3-ab, & e_1e_2&=c\overline{e_3}.
		\end{aligned}
	\end{equation*}
	The dual basis $\{e_0', e_1', e_2', e_3'\}$ of $\{1, e_1, e_2, e_3\}$ with respect to the
	trace pairing satisfies $\mathcal{O}^{\sharp}\supseteq\mathcal{O}$ and
	$\tr(e_0')=1$, $\tr(e_1')=\tr(e_2')=\tr(e_3')=0$. Explicitly, we have
	\begin{equation*}
		\begin{aligned}
			de_0'&=d-2(abc+rst)+(ar+st)e_1+(bs+rt)e_2+(ct+rs)e_3,\\
			de_1'&=ar+st-2ae_1-te_2-se_3,\\
			de_2'&=bs+rt-te_1-2be_2-re_3,\\
			de_3'&=ct+rs-se_1-re_2-2ce_3.
		\end{aligned}
	\end{equation*}
	These elements satisfy the norm and trace relations
	\begin{equation*}
		\begin{aligned}
			\n(de_1')&=da, & \tr(de_2'\overline{de_3'})&=dr,\\
			\n(de_2')&=db, & \tr(de_3'\overline{de_1'})&=ds,\\
			\n(de_3')&=dc, & \tr(de_1'\overline{de_2'})&=dt.
		\end{aligned}
	\end{equation*}
	Since $\mathcal{O}$ contains a $\mathbb{Q}$-basis of $Q$, the elements
	$\{e_1', e_2', e_3'\}$ form a $\mathbb{Q}$-basis for the trace-zero subspace $Q^0$.
	
	The correspondence between orders and quadratic forms preserves several important properties.
	
	\begin{theorem}\label{theoremforcliffordalgebras}
		Let $\mathcal{O}$ be a quaternion order and $f_{\mathcal{O}}$ its associated ternary
		quadratic form. Then:
		\begin{itemize}
			\item[(1)] $f_{\mathcal{O}}$ is positive definite if and only if $Q$ (equivalently, $\mathcal{O}$)
			is a definite quaternion algebra.
			\item[(2)] $Q$ ramifies at a prime $p$ if and only if $f_{\mathcal{O}}$ is anisotropic
			over $\mathbb{Q}_p$.
		\end{itemize}
	\end{theorem}
	
	We now characterize the genera of the quadratic forms $f_{S^{0}}$ and $f_{\mathcal{O}}$
	associated to orders of a given level.

	\begin{prop}\label{prop:genus_S0}
		Let $\mathcal{O}\subseteq Q_{N_0}$ be an order of level $(N_1,N_2)$. Then the ternary
		quadratic form $f_{S^0}$ associated to $S = \mathbb{Z} + 2\mathcal{O}$ satisfies:
		\begin{itemize}
			\item $d_{f_{S^0}}=16(N_1N_2)^2$,
			\item $N_{f_{S^0}}=4N_1N_2$ (the level of $f_{S^0}$),
			\item $f_{S^0}$ is anisotropic over $\mathbb{Q}_p$ if and only if $p\mid N_1$.
		\end{itemize}
		The genus to which $f_{S^0}$ belongs is denoted by $G_{f_{S^0}}$.
	\end{prop}
	
	\begin{proof}
		By the local-global principle, it suffices to analyze the local behavior of $f_{S^0}$
		at each prime $p$. We examine several cases based on the $p$-adic valuations of $N_1$
		and $N_2$.
		
		\medskip
		\noindent\textbf{Case 1: $p^{2u+1}\parallel N_1$ for an odd prime $p$.}
		In this case, $f_{S^0}$ is locally equivalent at $p$ to
		\begin{equation*}
			f_{S^0}\underset{p}{\sim}-\epsilon x^2-p^{2u+1}y^2+\epsilon p^{2u+1}z^2,
		\end{equation*}
		where $\epsilon$ is a quadratic non-residue modulo $p$, and $\underset{p}{\sim}$ denotes
		local equivalence at $p$. One verifies that $v_p(N_{f_{S^0}})=2u+1$ and $v_p(d_{f_{S^0}})=2(2u+1)=4u+2$.
		Moreover, $f_{S^0}$ is anisotropic over $\mathbb{Q}_p$ since it has no nontrivial zeros
		over $\mathbb{Q}_p$.
		
		\medskip
		\noindent\textbf{Case 2: $2^{2u+1}\parallel N_1$.}
		We have the local equivalence
		\begin{equation*}
			f_{S^0}\underset{2}{\sim}3x^2-2^{2u+3}(y^2+z^2+yz).
		\end{equation*}
		Again, one can verify that $v_2(N_{f_{S^0}})=2u+3$, $v_2(d_{f_{S^0}})=2(2u+3)$, and
		$f_{S^0}$ is anisotropic over $\mathbb{Q}_2$.
		
		\medskip
		\noindent\textbf{Case 3: $p^{v}\parallel N_2$ for an odd prime $p$.}
		In this case,
		\begin{equation*}
			f_{S^0}\underset{p}{\sim}-x^2-p^{v}yz.
		\end{equation*}
		Here $v_p(N_{f_{S^0}})=v$ and $v_p(d_{f_{S^0}})=2v$. The form is isotropic over
		$\mathbb{Q}_p$ (since it contains the hyperbolic plane $-p^v yz$).
		
		\medskip
		\noindent\textbf{Case 4: $2^{v}\parallel N_2$.}
		We have
		\begin{equation*}
			f_{S^0}\underset{2}{\sim}-x^2-2^{v+2}yz.
		\end{equation*}
		Similarly, $v_2(N_{f_{S^0}})=v+2$ and $v_2(d_{f_{S^0}})=2(v+2)$, and $f_{S^0}$ is isotropic
		over $\mathbb{Q}_2$.
		
		\medskip
Multiplication by a unit does not change its integral
isometry class; this permits the displayed sign convention.
Combining these local computations, we conclude that:
		\begin{itemize}
			\item $d_{f_{S^0}} = \prod_p p^{v_p(d_{f_{S^0}})} = 16(N_1N_2)^2$,
			\item $N_{f_{S^0}} = \prod_p p^{v_p(N_{f_{S^0}})} = 4N_1N_2$,
			\item $f_{S^0}$ is anisotropic at $p$ if and only if $p \mid N_1$.
		\end{itemize}
	\end{proof}

	\begin{prop}\label{prop:genus_O}
		Let $\mathcal{O}\subseteq Q_{N_0}$ be an order of level $(N_1,N_2)$. Then the ternary
		quadratic form $f_{\mathcal{O}}$ satisfies:
		\begin{itemize}
			\item $d_{f_{\mathcal{O}}}=N_1N_2$,
			\item $N_{f_{\mathcal{O}}}=4N_1N_2$ (the level of $f_{\mathcal{O}}$),
			\item $f_{\mathcal{O}}$ is anisotropic over $\mathbb{Q}_p$ if and only if $p\mid N_1$.
		\end{itemize}
		The genus to which $f_{\mathcal{O}}$ belongs is denoted by $G_{f_\mathcal{O}}$.
	\end{prop}
	
	\begin{proof}
		As in the proof of Proposition~\ref{prop:genus_S0}, we analyze the local behavior
		of $f_{\mathcal{O}}$ at each prime dividing $N_1N_2$.
		For a prime $p$, put $g=v_p(N)$ and $u_p=N/p^g$.

		\medskip
		\noindent\textbf{Case 1: $p^{2u+1}\parallel N_1$ for an odd prime $p$.}
		We have the local equivalence
		\begin{equation*}
			f_{\mathcal{O}}\underset{p}{\sim}u_p(-\epsilon p^{2u+1}x^2-y^2+\epsilon z^2),
		\end{equation*}
		where $\epsilon$ is a quadratic non-residue modulo $p$. One verifies that
		$v_p(N_{f_{\mathcal{O}}})=2u+1$, $v_p(d_{f_{\mathcal{O}}})=2u+1$, and
		$f_{\mathcal{O}}$ is anisotropic over $\mathbb{Q}_p$.
		
		\medskip
		\noindent\textbf{Case 2: $2^{2u+1}\parallel N_1$.}
		We have
		\begin{equation*}
			f_{\mathcal{O}}\underset{2}{\sim}u_p(3\cdot2^{2u+1}x^2-(y^2+z^2+yz)).
		\end{equation*}
		Similarly, $v_2(N_{f_{\mathcal{O}}})=2u+3$, $v_2(d_{f_{\mathcal{O}}})=2u+1$, and
		$f_{\mathcal{O}}$ is anisotropic over $\mathbb{Q}_2$.
		
		\medskip
		\noindent\textbf{Case 3: $p^{v}\parallel N_2$ for an odd prime $p$.}
		In this case,
		\begin{equation*}
			f_{\mathcal{O}}\underset{p}{\sim}u_p(-p^{v}x^2-yz).
		\end{equation*}
		Here $v_p(N_{f_{\mathcal{O}}})=v$, $v_p(d_{f_{\mathcal{O}}})=v$, and $f_{\mathcal{O}}$
		is isotropic over $\mathbb{Q}_p$.
		
		\medskip
		\noindent\textbf{Case 4: $2^{v}\parallel N_2$.}
		We have
		\begin{equation*}
			f_{\mathcal{O}}\underset{2}{\sim}u_p(-2^{v}x^2-yz).
		\end{equation*}
		Similarly, $v_2(N_{f_{\mathcal{O}}})=v+2$, $v_2(d_{f_{\mathcal{O}}})=v$, and
		$f_{\mathcal{O}}$ is isotropic over $\mathbb{Q}_2$.
		
		\medskip
		$u_p$ does not change level valuations or anisotropy.
		Combining these local computations yields the stated discriminant, level, and
		ramification properties of $f_{\mathcal{O}}$.
	\end{proof}
	
	\begin{prop}\label{propforO0andOat2}
		Let $\mathcal{O}$ be an order of level $(N_1,N_2)$. Let $\phi_{2N_1N_2}=\phi_{p_1}\circ\cdots\circ\phi_{p_{e(2N_1N_2)}}$,
		where $p_1,\ldots,p_{e(2N_1N_2)}$ are the distinct prime divisors of $2N_1N_2$, and
		$\phi_p$ denotes the Lehman correspondence. Then
		\begin{equation*}
			\phi_{2N_1N_2}(f_{S^0})=f_{\mathcal{O}}.
		\end{equation*}
	\end{prop}
	
	\begin{proof}
		The proof follows the same strategy as~\cite[Proposition 6.1]{LZ26}. When $p$ is an
		odd prime, the Lehman correspondence $\phi_p$ can be applied with $p$ replaced by
		$p^{2u+1}$ (if $p^{2u+1} \parallel N_1$) or $p^v$ (if $p^v \parallel N_2$), and the
		arguments of~\cite{LZ26} carry over verbatim in these cases.
		
		For the prime $2$, if $2 \nmid N_1N_2$, then the Watson transformation $\lambda_4$
		(which is equivalent to $\phi_2$) provides a bijection between the genera containing
		$f_{\mathcal{O}}$ and $f_{S^0}$; see~\cite[Section 4.1]{LZ26} for details. When
		$2 \mid N_1N_2$, the same substitution principle applies as for odd primes.
		
		Applying these transformations successively for all primes dividing $2N_1N_2$ yields
		the desired equality $\phi_{2N_1N_2}(f_{S^0})=f_{\mathcal{O}}$.
	\end{proof}
	
	The relationship between orders and quadratic forms can be summarized in the following
	commutative diagram.
	\begin{theorem}\label{thm:commutative_diagram}
The following diagram of sets of classes commutes:
\[
\begin{tikzcd}[row sep=large,column sep=large]
\{[\mathcal O_\mu]\}
  \ar[r,"{\mathcal O\mapsto f_{\mathcal O}}"]
  \ar[dr,"M_1"']
 &G_{f_\mathcal{O}}\\
 &G_{f_{S^0}}\ar[u,"\phi_{2N}"']
\end{tikzcd}
\]
All three maps are bijections. The inverse of the upper arrow
is the even Clifford construction $C_0$.
\end{theorem}
\begin{proof}
The upper arrow and its inverse are furnished by
Theorem~\ref{Lem}; the bijection $M_1$ is
Proposition~\ref{propforO0andOat2}.
The equality of classes
$\phi_{2N}(M_1(\mathcal O))=[f_{\mathcal O}]$ is the explicit
isometry in that proposition. Compatibility with localization
restricts the maps to the indicated genera.
\end{proof}

	As a consequence of Theorem~\ref{Lem} and Theorem~\ref{theoremforcliffordalgebras},
	together with~\cite[Proposition 2.5]{LZ26}, we obtain the following corollary relating
	automorphism groups.
	
	\begin{cor}\label{corollaryforcommutativediagrams3}
The map $\mathcal O\mapsto(\mathbb Z+2\mathcal O)^0$ with its
norm induces a bijection between types of orders of level
$(N_1,N_2)$ and classes in $G_{f_{S^0}}$.
		Moreover, the automorphism groups are related by
		\begin{equation*}
			2|\Aut(\mathcal{O})|=|\Aut(f_{S^0})|.
		\end{equation*}
	\end{cor}
	
	\begin{proof}
		The bijection follows from the bijectivity of the map $C_0$ in Theorem~\ref{Lem}
		combined with the Lehman correspondence in Proposition~\ref{propforO0andOat2}.
		
		For the automorphism count, by the preceding parts of the theorem and~\cite[Remark 22.4.3]{Voi21} we have for an order $\mathcal{O}$ of level $(N_1,N_2)$, 
\begin{equation*}
2|\operatorname{Aut}(\mathcal{O})|=|\operatorname{Aut}(f_{\mathcal{O}})|=|\operatorname{Aut}(\phi_{2N_1N_2}(f_{\mathcal{O}}))|=|\operatorname{Aut}(f_{S^0})|.
\end{equation*}
	\end{proof}

	\subsection{Counting zeros in orders}
	
	In this subsection, we establish a precise relationship between counting zeros of quadratic
	polynomials in quaternion orders and counting representations by their associated ternary
	quadratic forms. Specifically, we show that computing
	\begin{equation*}
		\sum_{\mu=1}^{T_{N_1,N_2}}\frac{\rho_{\mathcal{O}_\mu}(n,r)}{|\Aut(\mathcal{O}_\mu)|}
	\end{equation*}
	is equivalent to computing
	\begin{equation*}
		\sum\limits_{f\in G_{f_{S^0}}}\frac{R_{f}(4n-r^2)}{|\Aut(f)|},
	\end{equation*}
	where $\rho_{\mathcal{O}}(n,r)$ denotes the number of solutions to $x^2 - rx + n = 0$
	in an order $\mathcal{O}$, and $R_f(m)$ denotes the number of integral representations
	of $m$ by the ternary form $f$.
	
	\begin{prop}\label{R}
		Let $\mathcal{O}$ be an order of level $(N_1,N_2)$ in $Q_{N_0}$, and let
		$S = \mathbb{Z} + 2\mathcal{O}$ with associated quadratic form $f_{S^0}$. Then
		\begin{equation}\label{eq:representation_identity}
			R_{f_{S^{0}}}(4n-r^2)=\rho_\mathcal{O}(n,r),
		\end{equation}
		where $\rho_\mathcal{O}(n,r)$ counts the number of elements $\alpha \in \mathcal{O}$
		satisfying $\n(\alpha) = n$ and $\tr(\alpha) = r$.
	\end{prop}
	
	\begin{proof}
		The quadratic form $f_{S^0}$ is given by
		\begin{align*}
			f_{S^{0}}(x,y,z)
			&= \n(x(2\alpha_1)+y(2\alpha_2)+z(2\alpha_3-1))\\
			&= 4\n(\alpha_1)x^2+4\n(\alpha_2)y^2+(4\n(\alpha_3)-1)z^2+4\tr(\alpha_2\overline{\alpha_3})yz+4\tr(\alpha_1\overline{\alpha_3})xz
			+4\tr(\alpha_1\overline{\alpha_2})xy.
		\end{align*}
		
		We now analyze the equation $f_{S^0}(x,y,z) = m$ according to $m \bmod 4$.
		
		\medskip
		\noindent\textbf{Case 1: $m \equiv 1,2 \pmod{4}$.}
		
		Since the coefficients of $x^2$ and $y^2$ are divisible by $4$ and the coefficient
		of $z^2$ is $\equiv -1 \pmod{4}$, we have
		\begin{equation*}
			f_{S^{0}}(x,y,z) \equiv -(4\n(\alpha_3)-1)z^2 \equiv -z^2 \pmod{4}.
		\end{equation*}
		Thus $f_{S^0}(x,y,z) \equiv -z^2 \pmod{4}$, which can only be $\equiv 0$ or $3 \pmod{4}$.
		Therefore, $R_{f_{S^{0}}}(m)=0$ for $m \equiv 1,2 \pmod{4}$.
		
		\medskip
		
		\noindent\textbf{Case 2: $m = 4n - r^2 \equiv 0 \pmod{4}$ (i.e., $r$ even).}
		
		When $m \equiv 0 \pmod{4}$, we must have $z \equiv 0 \pmod{2}$. We establish a bijection
		between solutions $(x,y,z) \in \mathbb{Z}^3$ to $f_{S^0}(x,y,z) = 4n - r^2$ and elements
		$\alpha \in \mathcal{O}$ with $\n(\alpha) = n$ and $\tr(\alpha) = r$.
		
		\medskip
		\noindent\emph{Direction 1: From $\mathcal{O}$ to $f_{S^0}$.}
		Let $\alpha = x\alpha_1 + y\alpha_2 + z\alpha_3 + t \in \mathcal{O}$ with $\n(\alpha) = n$
		and $\tr(\alpha) = r$. Since $\tr(\alpha) = z + 2t = r$, we have $t = \frac{r-z}{2} \in \mathbb{Z}$
		(as $r$ and $z$ are both even). Let $\gamma = x\alpha_1 + y\alpha_2 + z\alpha_3 \in \mathcal{O}^0$.
		
		Using the identity $\tr(\gamma\overline{\delta}) = \n(\gamma+\delta) - \n(\gamma) - \n(\delta)$
		with $\delta = \frac{r-z}{2}$, we have
		\begin{align*}
			n = \n(\alpha)
			&= \n\left(\gamma + \frac{r-z}{2}\right)\\
			&= \n(\gamma) + \frac{r-z}{2}\tr(\gamma) + \frac{(r-z)^2}{4}\\
			&= \n(\gamma) + \frac{(r-z)z}{2} + \frac{(r-z)^2}{4}\\
			&= \n(\gamma) + \frac{r^2 - z^2}{4},
		\end{align*}
		where we used $\tr(\gamma) = z$ (since $\tr(\alpha_i) = 0$ for $i=1,2$ and $\tr(\alpha_3) = 1$).
		
		Now, observe that $x(2\alpha_1) + y(2\alpha_2) + z(2\alpha_3-1) = 2\gamma - z$.
		Applying the norm identity again:
		\begin{align*}
			f_{S^0}(x,y,z)
			&= \n(2\gamma - z)\\
			&= \n(2\gamma) + \tr(2\gamma \cdot \overline{(-z)}) + \n(z)\\
			&= 4\n(\gamma) - 2z\tr(\gamma) + z^2\\
			&= 4\n(\gamma) - z^2.
		\end{align*}
		
		Substituting $\n(\gamma) = n - \frac{r^2 - z^2}{4}$, we obtain
		\begin{equation*}
			f_{S^0}(x,y,z) = 4\left(n - \frac{r^2 - z^2}{4}\right) - z^2 = 4n - r^2.
		\end{equation*}
		
		\medskip
		\noindent\emph{Direction 2: From $f_{S^0}$ to $\mathcal{O}$.}
		Conversely, given $(x,y,z) \in \mathbb{Z}^3$ with $z$ even and $f_{S^0}(x,y,z) = 4n - r^2$,
		let $\gamma = x\alpha_1 + y\alpha_2 + z(2\alpha_3-1) \in \mathcal{O}^0$. From
		$f_{S^0}(x,y,z) = 4\n(\gamma) - z^2 = 4n - r^2$, we get
		$\n(\gamma) = n - \frac{r^2 - z^2}{4}$.
		
		Define $\alpha = \gamma + \frac{r-z}{2} \in \mathcal{O}$ (noting that
		$t = \frac{r-z}{2} \in \mathbb{Z}$ since $r$ and $z$ are both even). Then
		$\tr(\alpha) = z + 2t = r$ and the above calculation shows $\n(\alpha) = n$.
		
		This establishes the bijection for $r, z$ both even.
		
		\medskip
		\noindent\textbf{Case 3: $m = 4n - r^2 \equiv 3 \pmod{4}$ (i.e., $r$ odd).}
		
		In this case, $r$ is odd and $z$ must be odd. The argument proceeds similarly with
		$t = \frac{r-z}{2} \in \mathbb{Z}$ (since $r$ and $z$ have the same parity), and the
		same calculation shows $f_{S^0}(x,y,z) = 4n - r^2$.
		
		\medskip
		The converse direction follows by reversing the correspondence: given $(x,y,z) \in \mathbb{Z}^3$
		with $f_{S^0}(x,y,z) = 4n - r^2$, the element
		$\alpha = x\alpha_1 + y\alpha_2 + z\alpha_3 + \frac{r-z}{2} \in \mathcal{O}$ satisfies
		$\n(\alpha) = n$ and $\tr(\alpha) = r$. This completes the proof of~\eqref{eq:representation_identity}.
	\end{proof}
	
	Combining Corollary~\ref{corollaryforcommutativediagrams3} with Proposition~\ref{R}, we
	obtain the main result of this subsection.
	
	\begin{theorem}\label{thm:counting_equivalence}
		Let $\{\mathcal{O}_\mu\}_{\mu=1,\ldots,T_{N_1,N_2}}$ be a complete set of representatives
		for the isomorphism classes of orders of level $(N_1,N_2)$ in $Q_{N_0}$. Then
		\begin{equation}\label{eq:counting_formula}
			\sum_{\mu=1}^{T_{N_1,N_2}}\frac{\rho_{\mathcal{O}_\mu}(n,r)}{|\Aut(\mathcal{O}_\mu)|}
			=2\sum\limits_{f\in G_{f_{S^0}}}\frac{R_{f}(4n-r^2)}{|\Aut(f)|},
		\end{equation}
		where the sum on the right is over equivalence classes of forms in the genus
		$G_{f_{S^0}}$.
	\end{theorem}
	
	\begin{proof}
		By Corollary~\ref{corollaryforcommutativediagrams3}, the map $M_1: \mathcal{O}_\mu \mapsto f_{S_\mu^0}$
		(where $S_\mu = \mathbb{Z} + 2\mathcal{O}_\mu$) establishes a bijection between isomorphism
		classes of orders $\{\mathcal{O}_\mu\}$ and equivalence classes of forms in
		$G_{f_{S^0}}$, with the automorphism groups related by
		\begin{equation*}
			|\Aut(f_{S_\mu^0})| = 2|\Aut(\mathcal{O}_\mu)|.
		\end{equation*}
		
		By Proposition~\ref{R}, for each $\mu$ we have
		\begin{equation*}
			R_{f_{S_\mu^0}}(4n-r^2) = \rho_{\mathcal{O}_\mu}(n,r).
		\end{equation*}
		
		Therefore,
		\begin{align*}
			\sum_{\mu=1}^{T_{N_1,N_2}}\frac{\rho_{\mathcal{O}_\mu}(n,r)}{|\Aut(\mathcal{O}_\mu)|}
			&= \sum_{\mu=1}^{T_{N_1,N_2}}\frac{R_{f_{S_\mu^0}}(4n-r^2)}{|\Aut(\mathcal{O}_\mu)|}\\
			&= \sum_{\mu=1}^{T_{N_1,N_2}}\frac{2R_{f_{S_\mu^0}}(4n-r^2)}{|\Aut(f_{S_\mu^0})|}\\
			&= 2\sum\limits_{f\in G_{f_{S^0}}}\frac{R_{f}(4n-r^2)}{|\Aut(f)|},
		\end{align*}
		where the last equality uses the bijection from Corollary~\ref{corollaryforcommutativediagrams3}.
	\end{proof}
	
	\begin{remark}
		The factor of $2$ in equation~\eqref{eq:counting_formula} reflects the fact that the
		automorphism group of a quaternion order is smaller by a factor of $2$ compared to the
		automorphism group of its associated quadratic form, as explained in
		Corollary~\ref{corollaryforcommutativediagrams3}.
	\end{remark}
	
	\begin{remark}
		Theorem~\ref{thm:counting_equivalence} reduces the problem of counting zeros of
		quadratic polynomials in quaternion orders to the classical problem of counting
		representations by ternary quadratic forms, which can be attacked using modular forms
		and analytic methods. This connection is exploited in the subsequent sections.
	\end{remark}

	\subsection{Counting normalizers in orders}
	
	For $n>4$, the conditions $n\mid r$ and $4n-r^2\geq0$ imply $r=0$. When $n=4$, one can verify that if $\n(\alpha)=4$ and $\tr(\alpha)=\pm4$, then $\alpha=\pm2$ and $4\nmid c_{11}(\pm2)$. By Remark~\ref{remarkfortwosidedprincipal}, we have
	\begin{align*}
		T_{N_1,N_2}
		& =2^{-e(N_1N_2)}\sum_{\mu=1}^{T_{N_1,N_2}}2^{e(N_1N_2)} \\
		& =2^{-e(N_1N_2)}\sum_{\mu=1}^{T_{N_1,N_2}}m(\mathcal{O}_\mu)\card(\mathfrak{B}(\mathcal{O}_\mu)/\mathbb{Q}^\times) \\
		& =2^{-e(N_1N_2)}\sum_{\mu=1}^{T_{N_1,N_2}}m(\mathcal{O}_\mu)\sum_{n\parallel N_1N_2}\sum_{\substack{n\mid r\\r^2\leq4n}}\frac{\rho'_{\mathcal{O}_\mu}(n,r)}{\card(\mathcal{O}_\mu^\times)} \\
		& =2^{-1}\sum_{\substack{n\parallel N_1N_2}}\sum_{\substack{n\mid r\\r^2\leq4n}}\sum_{\mu=1}^{T_{N_1,N_2}}\frac{\rho'_{\mathcal{O}_\mu}(n,r)}{\card(\Aut(\mathcal{O}_\mu))}\\
		& =\sum_{\substack{n\parallel N_1N_2\\n\leq3}}\sum_{\substack{n\mid r\\r^2\leq4n}}\sum\limits_{f\in G_{f_{S^0}}}\frac{R_{f}(4n-r^2)}{|\Aut(f)|}+\sum_{\substack{n\parallel N_1N_2\\n\geq4}}\sum\limits_{f\in G_{f_{S^0}}}\frac{R'_{f}(4n)}{|\Aut(f)|}.
	\end{align*}
Here, for $f=f_{S^0}$ and $n\parallel N_1N_2$ with $n\ge4$,
we set $R'_f(4n)=\rho'_{\mathcal O}(n,0)$.
Thus the restricted count uses primitive normalizing elements,
consistently with Theorem~\ref{theoremforxingshu}.

	It remains to evaluate the second sum. The key is the following lemma relating representation numbers.
	
	\begin{lemma}\label{propforrepresentationnumbersatp}
		Let $f=(a,p^gb,p^gc,p^gr,p^gs,p^gt)$ be a ternary quadratic form of level $p^gN$ and discriminant $p^{2g}d$, where $p$ is an odd prime with $p\nmid N$ and $p\nmid d$. Then
		\begin{equation*}
			\card(\{(x,y,z)\in\mathbb{Z}^3: f(x,y,z)=p^gn,\, p^{g}\mid x\})=R_{\phi_p(f)}(n).
		\end{equation*}
		For $p=2$, we have
		\begin{equation*}
			\card(\{(x,y,z)\in\mathbb{Z}^3: f(x,y,z)=2^gn,\, 2^{g-1}\mid x\})=R_{\phi_2(f)}(n).
		\end{equation*}
	\end{lemma}
	
	\begin{proof}
		The proof is analogous to that of~\cite[Proposition 6.6]{LZ26}, with $p$ replaced by $p^{g}$ and variables $x$ and $z$ exchanged.
	\end{proof}
	
	\begin{lemma}\label{lemmaforxingshu1}
		Let $\mathcal{O}$ be an order of level $(N_1,N_2)$ and $n\parallel N_1N_2$ with $n\geq4$. Set
		\begin{equation*}
			R_\mathcal{O}(n)=\{\alpha\in\mathcal{O}:\alpha\text{ is primitive}, \alpha\in N_Q(\mathcal{O}),\, \n(\alpha)=n\}.
		\end{equation*}
		Then every such $\alpha$ has trace zero, and
		\begin{equation*}
			|R_\mathcal{O}(n)|=R_{\phi_{n}(f_{S^{0}})}(4^{\delta_{2\nmid n}}),
		\end{equation*}
		where $\delta_{2\nmid n}=1$ if $2\nmid n$ and otherwise $0$. Consequently, the corresponding weighted sums can be computed
with the transformed genus, with unchanged automorphism weights:
		\begin{equation*}
			\sum_{\substack{n\parallel N_1N_2\\n\geq4}}\sum\limits_{f\in G_{f_{S^0}}}\frac{R'_{f}(4n)}{|\Aut(f)|}=\sum_{\substack{n\parallel N_1N_2\\n\geq4}}\sum\limits_{f\in G_{f_{S^0}}}\frac{R_{\phi_{n}(f)}(4^{\delta_{2\nmid n}})}{|\Aut(\phi_{n}(f))|}.
		\end{equation*}
	\end{lemma}
	
	\begin{proof}
		We first treat the case where $p$ is an odd prime with $p^{2u+1}\parallel N_1$.
		
		By~\cite[Proposition 6.1]{LZ26} (with $p$ replaced by $p^{2u+1}$), we have
		\begin{equation*}
			\mathcal{O}=\mathbb{Z}e_0'+\mathbb{Z}e_1'+\mathbb{Z}p^{2u+1}e_2'+\mathbb{Z}p^{2u+1}e_3',
		\end{equation*}
		where $\tr(e_1')=\tr(e_2')=\tr(e_3')=0$, $\n(e_1')=a$ with $p\nmid a$ and $2\nmid a$, and the bilinear forms satisfy
		\begin{equation*}
			\n(p^{2u+1}e_2')=p^{2u+1}b,\quad \n(p^{2u+1}e_3')=p^{2u+1}c,\quad \tr(p^{2u+1}e_2'\overline{p^{2u+1}e_3'})=p^{2u+1}r,
		\end{equation*}
		with similar relations for the remaining inner products. This yields
		\begin{equation*}
			\mathcal{O}^0=\mathbb{Z}e_1'+\mathbb{Z}p^{2u+1}e_2'+\mathbb{Z}p^{2u+1}e_3',\quad
			S^{0}=\mathbb{Z}e_1'+\mathbb{Z}2p^{2u+1}e_2'+\mathbb{Z}2p^{2u+1}e_3',
		\end{equation*}
		and the quadratic forms
		\begin{align*}
			f_{S^{0}}(x,y,z)&=ax^2+4p^{2u+1}by^2+4p^{2u+1}cz^2+4p^{2u+1}ryz+2p^{2u+1}sxz+2p^{2u+1}txy,\\
			\phi_p(f_{S^{0}})(x,y,z)&=p^{2u+1}ax^2+4by^2+4cz^2+4ryz+2p^{2u+1}sxz+2p^{2u+1}txy.
		\end{align*}
		
		Since $p\nmid a=\n(e_1')$, we have $p\nmid c_{11}(e_1')$, while $p^{2u+1}\mid c_{11}(p^{2u+1}e_i')$ for $i=2,3$. By Lemma~\ref{propforrepresentationnumbersatp},
		\begin{equation*}
			\card(\{(x,y,z)\in\mathbb{Z}^3: f_{S^{0}}(x,y,z)=4n,\, p^{2u+1}\mid x\})=R_{\phi_p(f_{S^{0}})}(4n/p^{2u+1}).
		\end{equation*}
		
		We establish the bijection between such triples and normalizers. If $f_{S^{0}}(x,y,z)=4n$ with $p^{2u+1}\mid x$, then $2\mid x$. Setting $\alpha=(x/2)e_1'+yp^{2u+1}e_2'+zp^{2u+1}e_3'$, the condition $p^{2u+1}\mid x$ implies $p^{2u+1}\mid c_{11}(\alpha)$, so $\alpha\in N(\mathcal{O}_p)$ by Lemma~\ref{lemmaforprincipalideal}.
		
		Conversely, if $\alpha=xe_1'+yp^{2u+1}e_2'+zp^{2u+1}e_3'\in N(\mathcal{O}_p)$ with $\n(\alpha)=n$ and $\tr(\alpha)=0$, then $p^{2u+1}\mid c_{11}(\alpha)$. Since $p\nmid c_{11}(e_1')$, we have $p^{2u+1}\mid x$. Thus $\n(2\alpha)=4n$, giving the triple $(2x,y,z)$ with $f_{S^{0}}(2x,y,z)=4n$ and $p^{2u+1}\mid 2x$.
		
		The cases $p=2$ or $N_2=p^{v}$ are similar. The general case follows by induction on the number of prime divisors.
	\end{proof}
	
	Therefore, we obtain
	\begin{equation*}
		T_{N_1,N_2} =\sum_{\substack{n\parallel N_1N_2\\n\leq3}}\sum_{\substack{n\mid r\\r^2\leq4n}}\sum\limits_{f\in G_{f_{S^0}}}\frac{R_{f}(4n-r^2)}{|\Aut(f)|}+\sum_{\substack{n\parallel N_1N_2\\n\geq4}}\sum\limits_{f\in G_{f_{S^0}}}\frac{R_{\phi_{n}(f)}(4^{\delta_{2\nmid n}})}{|\Aut(\phi_{n}(f))|}.
	\end{equation*}

	\section{\texorpdfstring{The modified Hurwitz class number $H^{(N_1,N_2)}(D)$}{The modified Hurwitz class number H(N1,N2)(D)}}\label{Sec:4}
	
	When $N_1N_2$ is squarefree, Li, Skoruppa, and the second author~\cite{LSZ21} proved that for all Eichler orders with the same squarefree level in a definite quaternion algebra over $\mathbb{Q}$, a weighted sum of Jacobi theta series associated with these orders equals a Jacobi Eisenstein series whose Fourier coefficients are given by $H^{(N_1,N_2)}(4n-r^2)$.
	
	\begin{theorem}[\cite{LSZ21}, Main Theorem]\label{LSZ}
		Let $N$ and $F$ be two coprime squarefree positive integers, where $N$ has an odd number of prime factors. Denote by $T_{N,F}$ the type number of Eichler orders of level $F$ in $Q_N$, where $Q_N$ ramifies exactly at the primes dividing $N$. Choose a complete set of representatives $\mathcal{O}_{\mu}$ $(\mu=1,2,\ldots,T_{N,F})$ for these types of Eichler orders.
		Let
		\begin{equation*}
			\theta_{\mathcal{O_\mu}}=\sum_{\substack{n,r\in\mathbb{Z}\\4n-r^2\geq0}}\rho_{\mathcal{O_\mu}}(n,r)q^n\zeta^r,
		\end{equation*}
		where $\rho_{\mathcal{O_\mu}}(n,r)$ is the number of zeros of $x^2-rx+n$ in $\mathcal{O}_{\mu}$. Then we have
		\begin{equation}
			\sum\limits_{\mu=1}^{T_{N,F}}\frac{\theta_{\mathcal{O_\mu}}}{\card(\Aut(\mathcal{O}_\mu))}=2^{-e(NF)}\sum_{\substack{n,r\in\mathbb{Z}\\4n-r^2\geq0}}H^{(N,F)}(4n-r^2)q^n\zeta^r,
		\end{equation}
		where $e(NF)$ is the number of prime factors of $NF$, and $\card(\Aut(\mathcal{O}_{\mu}))$ is the number of elements in the group of automorphisms of $\mathcal{O}_{\mu}$.
	\end{theorem}
	
	We now present the generalized definition of $H^{(N_1,N_2)}(D)$.
	
	\begin{definition}\label{def:H}
		Let $N_1=p_1^{2u_1+1}\cdots p_w^{2u_w+1}$, where $p_1,\ldots,p_w$ are distinct primes, $u_1,\ldots,u_w$ are nonnegative integers, and $w$ is odd. Let $N_2$ be a positive integer with $\gcd(N_1,N_2)=1$.
		
		For any negative discriminant $-D$, let $f=f_{N_1,N_2}$ be the largest positive integer divisible only by primes dividing $N_1N_2$ such that $f_{N_1,N_2}^2\mid D$ and $-D/f_{N_1,N_2}^2$ remains a negative discriminant. Denote by $f_p$ the exact $p$-power dividing $f_{N_1,N_2}$, and let $H(D/f^2_{N_1,N_2})$ be the Hurwitz class number. Denote $\chi_p=\left(\frac{-D/f^2_{N_1,N_2}}{p}\right)$, where $(\frac{\cdot}{p})$ denotes the Kronecker symbol. Define
		\begin{equation}\label{Hn1n2}
			H^{(N_1,N_2)}(D)=H(D/f^2_{N_1,N_2})\prod\limits_{p\mid N_1N_2}A_{p}(D;N_1,N_2),
		\end{equation}
		where the local factors $A_p(D;N_1,N_2)$ are given in Table~\ref{tab:local_factors}:
		\begin{table}[ht]
			\centering
			\caption{Local factors $A_p(D;N_1,N_2)$}
			\label{tab:local_factors}
			\begin{tabular}{lll}
				\toprule
				Condition & $p \mid N_1$ & $p \mid N_2$ \\
				\midrule
				$v_p(pf^2) < v_p(N_1N_2)$ and $p \mid (D/f^2)$ & $0$ & $0$ \\
				$v_p(pf^2) < v_p(N_1N_2)$ and $p \nmid (D/f^2)$ & $f_p^2(1-\chi_p)$ & $f_p^2(1+\chi_p)$ \\
				$v_p(pf^2) \geq v_p(N_1N_2)$ & $p^{v_p(N_1)-1}(1-\chi_p)$ & \eqref{eq:vp2odd},\,\eqref{eq:vp2even} \\
				\bottomrule
			\end{tabular}
		\end{table}
		
		For the condition $v_p(pf^2_{N_1,N_2})\geq v_p(N_1N_2)$ and $p\mid N_2$, there are two cases:
		
		When $v_p(N_2)$ is odd,
		\begin{equation}\label{eq:vp2odd}
			A_p(D;N_1,N_2)=\frac{2p^{\frac{v_p(N_2)+1}{2}}f_p-p^{v_p(N_2)-1}(p+1)-(2p^{\frac{v_p(N_2)-1}{2}}f_p-p^{v_p(N_2)-1}(p+1))\chi_p}{p-1};
		\end{equation}
		
		When $v_p(N_2)$ is even,
		\begin{equation}\label{eq:vp2even}
			A_{p}(D;N_1,N_2)=\frac{(p^{\frac{v_p(N_2)}{2}}f_p-p^{v_p(N_2)-1})(p+1)-(p^{\frac{v_p(N_2)}{2}-1}f_p-p^{v_p(N_2)-1})(p+1)\chi_p}{p-1}.
		\end{equation}
		
		Finally, set
		\begin{equation*}
			H^{(N_1,N_2)}(0)=\frac{N_1N_2}{12}\prod\limits_{p\mid N_1}\left(1-\frac{1}{p}\right)\prod\limits_{p\mid N_2}\left(1+\frac{1}{p}\right),
		\end{equation*}
		and $H^{(N_1,N_2)}(D)=0$ for every positive integer $D\equiv1,2\pmod{4}$.
	\end{definition}
	
	\begin{remark}
		When $N_1N_2$ is squarefree, we have $v_p(pf^2_{N_1,N_2})\geq v_p(N_1N_2)=1$ for all $p\mid N_1N_2$. This gives
		\begin{equation*}
			A_p(D;N_1,N_2)=\begin{cases}
				1-\chi_p & \text{if } p\mid N_1,\\[6pt]
				\dfrac{2pf_p-(p+1)-(2f_p-(p+1))\chi_p}{p-1} & \text{if } p\mid N_2,
			\end{cases}
		\end{equation*}
		which agrees with the definition in~\cite[(1)]{LSZ21}.
	\end{remark}
	
	This section establishes the following generalization of Theorem~\ref{LSZ}.
	\begin{theorem}\label{the:pre}
Let $(N_1,N_2)$ satisfy the hypotheses of Theorem~\ref{the:type},
and let $\mathcal O_\mu$, $1\le\mu\le T_{N_1,N_2}$, represent
the types of orders of level $(N_1,N_2)$.
For integers $n,r$ with $D=4n-r^2\ge0$, one has
\[
\sum_{\mu=1}^{T_{N_1,N_2}}
\frac{\rho_{\mathcal O_\mu}(n,r)}{|\Aut(\mathcal O_\mu)|}
 =2^{-e(N_1N_2)}H^{(N_1,N_2)}(D).
\]
Here $\rho_{\mathcal O}(n,r)$ counts elements of reduced norm
$n$ and reduced trace $r$.
\end{theorem}
	\subsection{Local representation densities}
	
	When $N_1N_2$ is not squarefree, explicitly determining Fourier coefficients of Jacobi Eisenstein series requires the Siegel--Weil formula for ternary quadratic forms. For a genus $G$ of positive ternary forms with discriminant $d_G$,
	\begin{equation*}
		\sum\limits_{f\in G}\frac{R_{f}(n)}{|\Aut(f)|}=4\pi M(G)\sqrt{\frac{n}{d_G}}\prod\limits_{p}d_{G,p}(n),
	\end{equation*}
	where the sum runs over class representatives, the product is over all primes, the mass is
	\begin{equation*}
		M(G)=\sum\limits_{f\in G}\frac{1}{|\Aut(f)|},
	\end{equation*}
	and $d_{G,p}(n)$ denotes the $p$-adic local representation density:
	\begin{equation*}
		d_{G,p}(n)=\lim_{t\to\infty}\frac{1}{p^{2t}}|\{(x,y,z)\in(\mathbb Z/p^t\mathbb Z)^3: f(x,y,z)\equiv n\pmod{p^t}\}|,
	\end{equation*}
	where $f$ is any form in $G$.
	
	We require explicit formulas for local densities; proofs appear in Appendix~\ref{App:B}.
	
	\begin{prop}[Anisotropic case, odd primes]\label{anisotropic}
		Let $u$ be a nonnegative integer, $p$ an odd prime with $\left(\frac{\epsilon}{p}\right)=-1$, and $p\nmid m$. Then
		\begin{equation*}
			d_{-\epsilon x^2+p^{2u+1}y^2-\epsilon p^{2u+1}z^2,p}(n)=
			\begin{cases}
				0 & \text{if }   n=mp^{2k+1},\, k<u,\\
				p^{k}\left(1-\left(\frac{-m}{p}\right)\right) & \text{if }  n=mp^{2k},\, k\leq u,\\
				p^{2u-k}\left(1+\frac{1}{p}\right) & \text{if }   n=mp^{2k+1},\, k\geq u,\\
				p^{2u-k}\left(1-\left(\frac{-m}{p}\right)\right)& \text{if }   n=mp^{2k},\, k>u.
			\end{cases}
		\end{equation*}
	\end{prop}
	
	\begin{prop}[Isotropic case, odd primes]\label{isotropic}
		Let $p$ be an odd prime and $p\nmid m$.
		\begin{enumerate}
			\item For $v_0$ even:
			\begin{equation*}
				d_{-x^2-p^{v_0}yz,p}(n)=
				\begin{cases}
					0 & \text{if }   n=mp^{2k+1},\, 2k+1<v_0,\\
					p^{k}\left(1+\left(\frac{-m}{p}\right)\right) & \text{if }  n=mp^{2k},\, 2k<v_0,\\
					p^{\frac{v_0}{2}-1}+p^{\frac{v_0}{2}}-p^{v_0-k-2}-p^{v_0-k-1} & \text{if }   n=mp^{2k+1},\, 2k+1>v_0,\\
					p^{\frac{v_0}{2}}+p^{\frac{v_0}{2}-1}+\left(\frac{-m}{p}\right)p^{v_0-k-1}-p^{v_0-k-1} & \text{if }   n=mp^{2k},\, 2k\geq v_0.
				\end{cases}
			\end{equation*}
			
			\item For $v_1$ odd:
			\begin{equation*}
				d_{-x^2-p^{v_1}yz,p}(n)=
				\begin{cases}
					0 & \text{if }  n=mp^{2k+1},\, 2k+1<v_1,\\
					p^{k}(1+(\frac{-m}{p})) & \text{if }   n=mp^{2k},\, 2k<v_1,\\
					2p^{\frac{v_1-1}{2}}-p^{v_1-k-2}-p^{v_1-k-1} & \text{if }   n=mp^{2k+1},\, 2k+1\geq v_1,\\
					2p^{\frac{v_1-1}{2}}-p^{v_1-k-1}+(\frac{-m}{p})p^{v_1-k-1}& \text{if }  n=mp^{2k},\, 2k>v_1.
				\end{cases}
			\end{equation*}
		\end{enumerate}
	\end{prop}
\begin{prop}[Anisotropic case at $2$]\label{anisotropic2}
Let $u\ge0$ and let $D>0$ with $D\equiv0,3\pmod4$.
Write $D=4^km$, where $-m$ is a quadratic discriminant whose
conductor is odd. Equivalently, either $m\equiv3\pmod4$, or
$m=4m_0$ with $m_0\equiv1,2\pmod4$.
Then
\[
d_{3x^2-2^{2u+3}(y^2+yz+z^2),2}(D)=
\begin{cases}
2^{k+1}(1-(\frac{-m}{2})),&m\equiv3\pmod4,\ k<u,\\
0,&m=4m_0,\ k<u,\\
2^{2u+1-k}(1-(\frac{-m}{2})),&m\equiv3\pmod4,\ k\ge u,\\
3\cdot2^{2u-k},&m=4m_0,\ k\ge u.
\end{cases}
\]
\end{prop}
\begin{prop}[Isotropic case at $2$]\label{isotropic2}
Write $D=4^km$ as in Proposition~\ref{anisotropic2}.
If $v_0\ge0$ is even, then
\[
d_{-x^2-2^{v_0+2}yz,2}(D)=
\begin{cases}
2^{k+1}(1+(\frac{-m}{2})),&m\equiv3\pmod4,\ 2k<v_0,\\
0,&m=4m_0,\ 2k<v_0,\\
3\cdot2^{v_0/2},&m\equiv7\pmod8,\ 2k\ge v_0,\\
3\cdot2^{v_0/2}-2^{v_0+1-k},
     &m\equiv3\pmod8,\ 2k\ge v_0,\\
3\bigl(2^{v_0/2}-2^{v_0-k-1}\bigr),
     &m=4m_0,\ 2k\ge v_0.
\end{cases}
\]
If $v_1\ge1$ is odd, then
\[
d_{-x^2-2^{v_1+2}yz,2}(D)=
\begin{cases}
2^{k+1}(1+(\frac{-m}{2})),&m\equiv3\pmod4,\ 2k<v_1+1,\\
0,&m=4m_0,\ 2k<v_1-1,\\
2^{(v_1+3)/2},&m\equiv7\pmod8,\ 2k\ge v_1+1,\\
2^{(v_1+3)/2}-2^{v_1+1-k},
     &m\equiv3\pmod8,\ 2k\ge v_1+1,\\
2^{(v_1+3)/2}-3\cdot2^{v_1-k-1},
     &m=4m_0,\ 2k\ge v_1-1.
\end{cases}
\]
In each occurrence of $m=4m_0$, one has $m_0\equiv1,2\pmod4$.
\end{prop}
	\subsection{Proof of the theorem}
	\begin{proof}[Proof of Theorem~\ref{the:pre}]
Put $N=N_1N_2$, $e=e(N)$, and
$G=G_{f_{S^0}}$.
By Corollary~\ref{corollaryforcommutativediagrams3} and
Proposition~\ref{R}, it suffices to prove
\[
\sum_{f\in G}\frac{R_f(D)}{|\Aut(f)|}
 =2^{-e-1}H^{(N_1,N_2)}(D).
\]
For $D=0$, this is the mass identity \eqref{eq:masstype},
together with $|\Aut(f_{S^0})|=2|\Aut(\mathcal O)|$.
Assume henceforth that $D>0$.

Let $N_i^{\mathrm{sf}}=\prod_{p\mid N_i}p$ for $i=1,2$,
and let $G^{\mathrm{sf}}$ be the corresponding squarefree genus.
The prime supports are unchanged, so the integer
$f=f_{N_1,N_2}$ and the quantities $f_p$ and $\chi_p$ in
Definition~\ref{def:H} are the same for both levels.
For $p\mid N$ put $\eta_p=2$ if $p=2$, and $\eta_p=1$
otherwise, and define the positive number
\[
c_p(D)=
\begin{cases}
\displaystyle\frac{\eta_p(p+1)}{f_p(p-\chi_p)},&p\mid N_1,\\[6pt]
\displaystyle\frac{\eta_p(p-1)}{f_p(p-\chi_p)},&p\mid N_2.
\end{cases}
\]
The four local density formulas give the division-free identities
\begin{align*}
d_{G,p}(D)&=c_p(D)A_p(D;N_1,N_2),\\
d_{G^{\mathrm{sf}},p}(D)
 &=c_p(D)A_p(D;N_1^{\mathrm{sf}},N_2^{\mathrm{sf}}).
\end{align*}
They hold for all three values $\chi_p=0,1,-1$.
For example, at an odd split prime with $\chi_p=0$,
\[
d_{-x^2-pyz,p}(D)
 =2-\frac{p+1}{pf_p}
 =\frac{p-1}{pf_p}
   \frac{2pf_p-p-1}{p-1},
\]
This also covers the case in which the quadratic discriminant
is divisible by $p$.
At $2$, the factor $\eta_2=2$ accounts for the norm lattice
$(\mathbb Z+2\mathcal O)^0$.

For $q\nmid N$, the two local norm lattices are isometric.
Moreover, the mass formula gives
\[
\frac{M(G)}{\sqrt{d_G}}
 =\frac{2^{-e-1}}{48}
   \prod_{p\mid N_1}(1-p^{-1})
   \prod_{p\mid N_2}(1+p^{-1})
 =\frac{M(G^{\mathrm{sf}})}{\sqrt{d_{G^{\mathrm{sf}}}}}.
\]
If a squarefree local factor vanishes, then either
$p\mid N_1$ and $\chi_p=1$, or
$p\mid N_2$, $f_p=1$ and $\chi_p=-1$.
In either case the factor at the original level vanishes as well.
Both sides of the desired identity are then zero.

Otherwise all squarefree local factors are nonzero, and the
Siegel--Weil formula yields
\begin{align*}
\sum_{f\in G}\frac{R_f(D)}{|\Aut(f)|}
&=\left(\sum_{f\in G^{\mathrm{sf}}}
                 \frac{R_f(D)}{|\Aut(f)|}\right)
\prod_{p\mid N}
\frac{A_p(D;N_1,N_2)}
     {A_p(D;N_1^{\mathrm{sf}},N_2^{\mathrm{sf}})}\\
&=2^{-e-1}H(D/f^2)\prod_{p\mid N}A_p(D;N_1,N_2).
\end{align*}
The last step uses the squarefree case of \cite{LSZ21}.
Definition~\ref{def:H} completes the proof.
\end{proof}

	\section{Proof of the type number formula}\label{Sec:5}
	
	This section establishes the type number formula. The key step is to compute
	\begin{equation*}
		\sum\limits_{f\in G_{f_{S^0}}}\frac{R_{\phi_{n}(f)}(4^{\delta_{2\nmid n}})}{|\Aut(\phi_{n}(f))|}.
	\end{equation*}
	We begin with local density calculations.
	\begin{prop}\label{propfortypenumberatp}
Let $p$ be odd, $u\ge0$, $v\ge1$, and
$\epsilon\in\mathbb Z_p^\times$ with $(\frac\epsilon p)=-1$.
For every $a\in\mathbb Z_p^\times$,
\[
d_{-\epsilon p^{2u+1}x^2+y^2-\epsilon z^2,p}(a)=1+p^{-1},
\qquad
d_{-p^vx^2-yz,p}(a)=1-p^{-1}.
\]
\end{prop}
\begin{proof}
Modulo $p$, the binary unramified norm represents a nonzero
element in $p+1$ ways, while $yz$ represents a nonzero element
in $p-1$ ways. The variable $x$ is arbitrary.
All these solutions are nonsingular and each lifts to $p^2$
solutions at each subsequent modulus. Division by $p^{2t}$
gives the stated densities.
\end{proof}
\begin{prop}\label{propfortypenumberat2}
Let $u\ge0$ and $v\ge1$. For every $a\in\mathbb Z_2^\times$,
\[
d_{3\cdot2^{2u+1}x^2-(y^2+yz+z^2),2}(a)=\frac32,
\qquad
d_{-2^vx^2-yz,2}(a)=\frac12.
\]
\end{prop}
\begin{proof}
For every $x$ the required binary value is odd.
Modulo $2^t$, an odd value has $3\cdot2^{t-1}$ representations
by $y^2+yz+z^2$ and $2^{t-1}$ representations by $yz$.
Multiplying by the $2^t$ choices of $x$ and dividing by $2^{2t}$
proves the result.
\end{proof}

	The following lemma relates representation numbers before and after the transformation $\phi_{n}$.
	\begin{lemma}\label{lemmaforxingshu2}
Let $N=N_1N_2$, $G=G_{f_{S^0}}$, and $n\parallel N$.
Then
\begin{align*}
\sum_{f\in G}
\frac{R_{\phi_n(f)}(4^{\delta_{2\nmid n}})}
     {|\Aut(\phi_n(f))|}
&=2^{-e(N)-1}H^{(N_1,N_2)}(4n)\\
&\quad\times\prod_{p\mid n}
\frac{1-(\frac{\Delta(-4n)}p)/p}{B_p(n)C_p(n)}.
\end{align*}
\end{lemma}
\begin{proof}
Let $G_n=\phi_n(G)$ and $m_n=4^{\delta_{2\nmid n}}$.
Lemma~\ref{propforrepresentationnumbersatp} gives
$M(G_n)=M(G)$ and
\[
d_{G_n}=\begin{cases}d_G/n,&2\nmid n,\\
                         d_G/(16n),&2\mid n.
          \end{cases}
\]
Consequently the ratio of the archimedean factors
$\sqrt{m_n/d_{G_n}}$ and $\sqrt{4n/d_G}$ is $1$ for odd $n$
and $2$ for even $n$.
At $q\nmid n$, the lattice is unchanged and the quadratic
form and represented integer are both multiplied by the same
$q$-adic unit. The local densities are therefore equal.

At an odd $p^g\parallel n$, write $n=p^gm$.
The local density ratio is
\[
\frac{d_{G_n,p}(m_n)}{d_{G,p}(4n)}=
\begin{cases}
p^{-(g-1)/2},&g\text{ odd},\\[2pt]
\displaystyle\frac{1-(\frac{-m}p)/p}
 {p^{g/2-1}(p+1)},&g\text{ even}.
\end{cases}
\]
Indeed, when $p\mid N_1$ the two densities are
$1+p^{-1}$ and $p^{(g-1)/2}(1+p^{-1})$.
For $p\mid N_2$ and odd $g$ they are
$1-p^{-1}$ and $p^{(g-1)/2}(1-p^{-1})$.
For even $g$ they are $1-p^{-1}$ and
$p^{g/2}(1+(\frac{-m}p)/p)$.
Here the unit in the transformed local form can be absorbed
into the represented unit, so Proposition~\ref{propfortypenumberatp}
applies without discarding any global unit factor.

Suppose $2^g\parallel n$. Incorporate the archimedean factor $2$
into the local ratio at $2$.
If $g$ is odd, Propositions~\ref{anisotropic2},
\ref{isotropic2} and \ref{propfortypenumberat2} give
\[
2\frac{d_{G_n,2}(1)}{d_{G,2}(4n)}=2^{-(g-1)/2}
\]
in both the division and split cases.
If $g=2s$ is even, then $2\mid N_2$.
Writing $n=2^{2s}m$ with $m$ odd, the original density is
\[
d_{G,2}(4n)=
\begin{cases}
3\cdot2^s,&m\equiv7\pmod8,\\
2^{s+1},&m\equiv3\pmod8,\\
3\cdot2^{s-1},&m\equiv1\pmod4,
\end{cases}
\]
and $d_{G_n,2}(1)=1/2$.
Thus, in all cases,
\[
2\frac{d_{G_n,2}(1)}{d_{G,2}(4n)}
 =\frac{1-(\frac{\Delta(-4n)}2)/2}{B_2(n)C_2(n)}.
\]
All denominators at the selected primes are positive.
The Siegel--Weil formula now compares the two genus averages.
Finally, Theorem~\ref{the:pre}, combined with
Corollary~\ref{corollaryforcommutativediagrams3}, identifies the
original average with $2^{-e(N)-1}H^{(N_1,N_2)}(4n)$.
\end{proof}

	\subsection{Proof of the type number formula}
	
	We now assemble the proof of Theorem~\ref{the:type}.
	
	\begin{proof}[Proof of Theorem~\ref{the:type}]
We first prove the class number formula. 
Every unit in a definite quaternion order has reduced norm $1$.
The possible traces are $-2,-1,0,1,2$, with $\pm1$ the only
units of trace $\pm2$. Hence
\[
|\mathcal O^\times|
 =2+\rho_{\mathcal O}(1,0)+2\rho_{\mathcal O}(1,1).
\]
Using the fibers of the ideal-class/type map and
\eqref{XINGSHU2}, followed by Theorem~\ref{the:pre}, gives
\begin{align*}
h_{N_1,N_2}
 &=\sum_\mu m(\mathcal O_\mu)\\
 &=2^{e(N_1N_2)-1}\sum_\mu
   \frac{2+\rho_{\mathcal O_\mu}(1,0)
                +2\rho_{\mathcal O_\mu}(1,1)}
        {|\Aut(\mathcal O_\mu)|}\\
 &=H^{(N_1,N_2)}(0)
   +\frac12H^{(N_1,N_2)}(4)+H^{(N_1,N_2)}(3).
\end{align*}
We next prove the type number formula.By equations~\eqref{XINGSHU0}, \eqref{XINGSHU1}, \eqref{XINGSHU2}, together with Corollary~\ref{corollaryforcommutativediagrams3}, Proposition~\ref{R}, Lemmas~\ref{lemmaforxingshu1} and~\ref{lemmaforxingshu2}, we obtain
		\begin{align*}
			T_{N_1,N_2}
			&= 2^{-e(N_1N_2)}\sum_{\mu=1}^{T_{N_1,N_2}}2^{e(N_1N_2)} \\
			&= 2^{-e(N_1N_2)}\sum_{\mu=1}^{T_{N_1,N_2}}m(\mathcal{O}_\mu)\cdot\card(\mathfrak{B}(\mathcal{O}_\mu)/\mathbb{Q}^\times) \\
			&= 2^{-e(N_1N_2)}\sum_{\mu=1}^{T_{N_1,N_2}}m(\mathcal{O}_\mu)\sum_{n\parallel N_1N_2}\sum_{\substack{n\mid r\\r^2\leq4n}}\frac{\rho'_{\mathcal{O}_\mu}(n,r)}{\card(\mathcal{O}_\mu^\times)} \\
			&= 2^{-1}\sum_{n\parallel N_1N_2}\sum_{\substack{n\mid r\\r^2\leq4n}}\sum_{\mu=1}^{T_{N_1,N_2}}\frac{\rho'_{\mathcal{O}_\mu}(n,r)}{\card(\Aut(\mathcal{O}_\mu))}\\
			&= \sum_{\substack{n\parallel N_1N_2\\n\leq3}}\sum_{\substack{n\mid r\\r^2\leq4n}}\sum\limits_{f\in G_{f_{S^0}}}\frac{R_{f}(4n-r^2)}{|\Aut(f)|} +\sum_{\substack{n\parallel N_1N_2\\n\geq4}}\sum\limits_{f\in G_{f_{S^0}}}\frac{R'_{f}(4n)}{|\Aut(f)|}\\
			&= \sum_{\substack{n\parallel N_1N_2\\n\leq3}}\sum_{\substack{n\mid r\\r^2\leq4n}}\sum\limits_{f\in G_{f_{S^0}}}\frac{R_{f}(4n-r^2)}{|\Aut(f)|} +\sum_{\substack{n\parallel N_1N_2\\n\geq4}}\sum\limits_{f\in G_{f_{S^0}}}\frac{R_{\phi_{n}(f)}(4^{\delta_{2\nmid n}})}{|\Aut(\phi_{n}(f))|}\\
			&= 2^{-e(N_1N_2)-1}\Bigg(\sum_{\substack{n\parallel N_1N_2\\n\leq3}}\sum_{\substack{n\mid r\\r^2\leq4n}}H^{(N_1,N_2)}(4n-r^2)  +\sum_{\substack{n\parallel N_1N_2\\n\geq4}}H^{(N_1,N_2)}(4n)\prod_{p\mid n}\frac{1-(\frac{\Delta(-4n)}{p})/p}{B_p(n)C_p(n)}\Bigg).
		\end{align*}
		
		For $n\leq 3$, direct verification shows
		\begin{equation*}
			\prod_{p\mid n}\frac{1-(\frac{\Delta(-4n)}{p})/p}{B_p(n)C_p(n)}=1.
		\end{equation*}
		This completes the proof of the type number formula.
	\end{proof}
	
	\begin{remark}
		The restriction $n\leq 3$ in the final step ensures that we can uniformly treat small values. For $n=1,2,3$, the discriminant $\Delta(-4n)$ satisfies $p\mid \Delta(-4n)$ for all $p\mid n$, which forces the Kronecker symbol to vanish appropriately.
	\end{remark}
	\subsection{Orders of type number one}

\begin{cor}\label{cor:type_one}
Under the hypotheses of Theorem~\ref{the:type}, one has
$T_{N_1,N_2}=1$ precisely for the following $27$ pairs:
\[
\begin{aligned}
& (2,1),(2,3),(2,5),(2,7),(2,9),(2,11),(2,15),(2,23),\\
& (3,1),(3,2),(3,4),(3,5),(3,8),(3,11),\\
& (5,1),(5,2),(5,4),(7,1),(7,3),(8,1),(8,5),\\
& (13,1),(30,1),(42,1),(70,1),(78,1),(168,1).
\end{aligned}
\]
\end{cor}
\begin{proof}
Put $N=N_1N_2$. If there is only one type, the mass identity
\eqref{eq:masstype} gives
\[
\frac1{|\Aut(\mathcal O)|}
 =\frac1{12}
 \prod_{p^{2u+1}\parallel N_1}\frac{p^{2u}(p-1)}2
 \prod_{p^v\parallel N_2}\frac{p^{v-1}(p+1)}2.
\]
Consequently a necessary condition is
\begin{equation}\label{eq:type_one_bound}
\prod_{p^{2u+1}\parallel N_1}\frac{p^{2u}(p-1)}2
\prod_{p^v\parallel N_2}\frac{p^{v-1}(p+1)}2\le12.
\end{equation}
All factors are at least $1$, except for the possible factor
$1/2$ at $2\parallel N_1$. Thus each individual factor is at
most $24$. In particular,
\[
p^{2u}(p-1)\le48\quad(p^{2u+1}\parallel N_1),\qquad
p^{v-1}(p+1)\le48\quad(p^v\parallel N_2).
\]
Every prime involved is at most $47$, and these inequalities
bound its exponent. There are therefore only finitely many
candidates. For each prime $p\le47$,
either no factor, a ramified factor $p^{2u+1}$ satisfying its
displayed bound, or a split factor $p^v$ satisfying its bound,
and retaining exactly the products satisfying
\eqref{eq:type_one_bound} and the parity condition.
This gives $136$ pairs with $N_1N_2\le690$. By Theorem~\ref{the:type}, we give exactly the $27$ pairs in the statement.
\end{proof}

\newpage
	\appendix
	
	\section{\texorpdfstring{$h_{N_1,N_2}$ and $T_{N_1,N_2}$ for orders of level $(N_1,N_2)$}{h\_N1,N2 and T\_N1,N2 for orders of level (N1,N2)}}\label{App:A}
	
	\begin{longtable}{@{}c c c c c c c c c c@{}}
		\caption{$h_{N_1,N_2}$ and $T_{N_1,N_2}$ for orders of level $(N_1,N_2)$}\label{tab:typenumber}\\
		\toprule[1.5pt]
		\multicolumn{1}{c}{$N_1N_2$} &
		\multicolumn{1}{c}{$N_1$} &
		\multicolumn{1}{c}{$N_2$} &
		\multicolumn{1}{c}{$h_{N_1,N_2}$} &
		\multicolumn{1}{c}{$T_{N_1,N_2}$} &
		\multicolumn{1}{c}{$N_1N_2$} &
		\multicolumn{1}{c}{$N_1$} &
		\multicolumn{1}{c}{$N_2$} &
		\multicolumn{1}{c}{$h_{N_1,N_2}$} &
		\multicolumn{1}{c}{$T_{N_1,N_2}$} \\
		\midrule[1pt]
		\endfirsthead
		
		\multicolumn{10}{@{}l@{}}{\small\textbf{Table \thetable\ (continued)}} \\
		\toprule[1.5pt]
		\multicolumn{1}{c}{$N_1N_2$} &
		\multicolumn{1}{c}{$N_1$} &
		\multicolumn{1}{c}{$N_2$} &
		\multicolumn{1}{c}{$h_{N_1,N_2}$} &
		\multicolumn{1}{c}{$T_{N_1,N_2}$} &
		\multicolumn{1}{c}{$N_1N_2$} &
		\multicolumn{1}{c}{$N_1$} &
		\multicolumn{1}{c}{$N_2$} &
		\multicolumn{1}{c}{$h_{N_1,N_2}$} &
		\multicolumn{1}{c}{$T_{N_1,N_2}$} \\
		\midrule[1pt]
		\endhead
		
		\midrule[0.5pt]
		\multicolumn{10}{r@{}}{\footnotesize\textit{Continued on next page}} \\
		\endfoot
		
		\bottomrule[1.5pt]
		\endlastfoot
		
		2 & 2 & 1 & 1 & 1 & 56  & 7 & 8 & 6 & 4 \\
		3 & 3 & 1 & 1 & 1 & 56  & 8 & 7 & 4 & 2 \\
		5 & 5 & 1 & 1 & 1 & 57  & 3 & 19  & 4 & 3 \\
		6 & 2 & 3 & 1 & 1 & 57  & 19  & 3 & 6 & 2 \\
		6 & 3 & 2 & 1 & 1 & 58  & 2 & 29  & 3 & 2 \\
		7 & 7 & 1 & 1 & 1 & 58  & 29  & 2 & 7 & 3 \\
		8 & 8 & 1 & 1 & 1 & 59  & 59  & 1 & 6 & 6 \\
		10  & 2 & 5 & 1 & 1 & 60  & 3 & 20  & 6 & 2 \\
		10  & 5 & 2 & 1 & 1 & 60  & 5 & 12  & 8 & 3 \\
		11  & 11  & 1 & 2 & 2 & 61  & 61  & 1 & 5 & 4 \\
		12  & 3 & 4 & 1 & 1 & 62  & 2 & 31  & 4 & 2 \\
		13  & 13  & 1 & 1 & 1 & 62  & 31  & 2 & 8 & 5 \\
		14  & 2 & 7 & 2 & 1 & 63  & 7 & 9 & 6 & 3 \\
		14  & 7 & 2 & 2 & 2 & 65  & 5 & 13  & 6 & 3 \\
		15  & 3 & 5 & 2 & 1 & 65  & 13  & 5 & 6 & 3 \\
		15  & 5 & 3 & 2 & 2 & 66  & 2 & 33  & 4 & 2 \\
		17  & 17  & 1 & 2 & 2 & 66  & 3 & 22  & 6 & 2 \\
		18  & 2 & 9 & 1 & 1 & 66  & 11  & 6 & 10  & 3 \\
		19  & 19  & 1 & 2 & 2 & 66  & 66  & 1 & 4 & 2 \\
		20  & 5 & 4 & 2 & 1 & 67  & 67  & 1 & 6 & 4 \\
		21  & 3 & 7 & 2 & 2 & 68  & 17  & 4 & 8 & 3 \\
		21  & 7 & 3 & 2 & 1 & 69  & 3 & 23  & 4 & 2 \\
		22  & 2 & 11  & 1 & 1 & 69  & 23  & 3 & 8 & 5 \\
		22  & 11  & 2 & 3 & 2 & 70  & 2 & 35  & 4 & 2 \\
		23  & 23  & 1 & 3 & 3 & 70  & 5 & 14  & 8 & 3 \\
		24  & 3 & 8 & 2 & 1 & 70  & 7 & 10  & 10  & 3 \\
		24  & 8 & 3 & 2 & 2 & 70  & 70  & 1 & 2 & 1 \\
		26  & 2 & 13  & 3 & 2 & 71  & 71  & 1 & 7 & 7 \\
		26  & 13  & 2 & 3 & 2 & 72  & 8 & 9 & 4 & 2 \\
		27  & 27  & 1 & 2 & 2 & 73  & 73  & 1 & 6 & 4 \\
		28  & 7 & 4 & 3 & 2 & 74  & 2 & 37  & 5 & 3 \\
		29  & 29  & 1 & 3 & 3 & 74  & 37  & 2 & 9 & 4 \\
		30  & 2 & 15  & 2 & 1 & 75  & 3 & 25  & 6 & 3 \\
		30  & 3 & 10  & 4 & 2 & 76  & 19  & 4 & 9 & 4 \\
		30  & 5 & 6 & 4 & 2 & 77  & 7 & 11  & 6 & 4 \\
		30  & 30  & 1 & 2 & 1 & 77  & 11  & 7 & 8 & 3 \\
		31  & 31  & 1 & 3 & 3 & 78  & 2 & 39  & 6 & 2 \\
		32  & 32  & 1 & 2 & 2 & 78  & 3 & 26  & 8 & 3 \\
		33  & 3 & 11  & 2 & 1 & 78  & 13  & 6 & 12  & 4 \\
		33  & 11  & 3 & 4 & 3 & 78  & 78  & 1 & 2 & 1 \\
		34  & 2 & 17  & 2 & 2 & 79  & 79  & 1 & 7 & 6 \\
		34  & 17  & 2 & 4 & 2 & 80  & 5 & 16  & 8 & 3 \\
		35  & 5 & 7 & 4 & 3 & 82  & 2 & 41  & 4 & 3 \\
		35  & 7 & 5 & 4 & 2 & 82  & 41  & 2 & 10  & 4 \\
		37  & 37  & 1 & 3 & 2 & 83  & 83  & 1 & 8 & 7 \\
		38  & 2 & 19  & 3 & 2 & 84  & 3 & 28  & 8 & 3 \\
		38  & 19  & 2 & 5 & 3 & 84  & 7 & 12  & 12  & 2 \\
		39  & 3 & 13  & 4 & 3 & 85  & 5 & 17  & 6 & 3 \\
		39  & 13  & 3 & 4 & 2 & 85  & 17  & 5 & 8 & 3 \\
		40  & 5 & 8 & 4 & 2 & 86  & 2 & 43  & 5 & 3 \\
		40  & 8 & 5 & 2 & 1 & 86  & 43  & 2 & 11  & 5 \\
		41  & 41  & 1 & 4 & 4 & 87  & 3 & 29  & 6 & 3 \\
		42  & 2 & 21  & 4 & 2 & 87  & 29  & 3 & 10  & 6 \\
		42  & 3 & 14  & 4 & 2 & 88  & 8 & 11  & 4 & 3 \\
		42  & 7 & 6 & 6 & 2 & 88  & 11  & 8 & 10  & 3 \\
		42  & 42  & 1 & 2 & 1 & 89  & 89  & 1 & 8 & 7 \\
		43  & 43  & 1 & 4 & 3 & 90  & 2 & 45  & 6 & 2 \\
		44  & 11  & 4 & 5 & 3 & 90  & 5 & 18  & 12  & 3 \\
		45  & 5 & 9 & 4 & 2 & 91  & 7 & 13  & 8 & 3 \\
		46  & 2 & 23  & 2 & 1 & 91  & 13  & 7 & 8 & 4 \\
		46  & 23  & 2 & 6 & 4 & 92  & 23  & 4 & 11  & 6 \\
		47  & 47  & 1 & 5 & 5 & 93  & 3 & 31  & 6 & 4 \\
		48  & 3 & 16  & 4 & 2 & 93  & 31  & 3 & 10  & 3 \\
		50  & 2 & 25  & 3 & 2 & 94  & 2 & 47  & 4 & 2 \\
		51  & 3 & 17  & 4 & 2 & 94  & 47  & 2 & 12  & 7 \\
		51  & 17  & 3 & 6 & 4 & 95  & 5 & 19  & 8 & 4 \\
		52  & 13  & 4 & 6 & 2 & 95  & 19  & 5 & 10  & 6 \\
		53  & 53  & 1 & 5 & 4 & 96  & 3 & 32  & 8 & 3 \\
		54  & 2 & 27  & 3 & 2 & 96  & 32  & 3 & 6 & 4 \\
		54  & 27  & 2 & 5 & 3 & 97  & 97  & 1 & 8 & 5 \\
		55  & 5 & 11  & 4 & 2 & 98  & 2 & 49  & 6 & 3 \\
		55  & 11  & 5 & 6 & 4 & 99  & 11  & 9 & 10  & 5 \\
	\end{longtable}
	\newpage
	
	\begin{longtable}{lcccc}
		\caption{$h_{p^{2u+1},N_2}$ and $T_{p^{2u+1},N_2}$}\label{tab:Boyd}\\
		\toprule[2pt]
		$p^{2u+1}N_2$ & $p^{2u+1}$ & $N_2$ & $h_{p^{2u+1},N_2}$ & $T_{p^{2u+1},N_2}$\\
		\midrule
		$5$ & $5$ & $1$ & $1$ & $1$\\
		$15$ & $3$ & $5$ & $2$ & $1$\\
		$27$ & $3^3$ & $1$ & $2$ & $2$\\
		$35$ & $5$ & $7$ & $4$ & $3$\\
		$125$ & $5^3$ & $1$ & $9$ & $7$\\
		$135$ & $3^3$ & $5$ & $10$ & $4$\\
		$189$ & $3^3$ & $7$ & $12$ & $6$\\
		$243$ & $3^5$ & $1$ & $14$ & $10$\\
		$250$ & $5^3$ & $2$ & $25$ & $9$\\
		$343$ & $7^3$ & $1$ & $25$ & $16$\\
		$405$ & $5$ & $3^4$ & $36$ & $11$\\
		$750$ & $5^3$ & $6$ & $100$ & $18$\\
		$972$ & $3^5$ & $4$ & $81$ & $25$\\
		$1000$ & $5^3$ & $8$ & $100$ & $28$ \\
		$1331$ & $11^3$ & $1$ & $102$ & $62$\\
		$2187$ & $3^7$ & $1$ & $122$ & $70$\\
		$3125$ & $5^5$ & $1$ & $209$ & $117$\\
		$4116$ & $7^3$ & $12$ & $588$ & $77$\\
		$16807$ & $7^5$ & $1$ & $1201$ & $625$\\
		$35152$ & $13^3$ & $2^4$ & $4056$ & $1027$\\
		$78125$ & $5^7$ & $1$ & $5209$ & $2667$\\
		$322102$ & $11^5$ & $2$ & $36603$ & $9272$\\
		$823543$ & $7^7$ & $1$ & $58825$ & $29584$\\
		\bottomrule[2pt]
	\end{longtable}
	
	\begin{remark}
		We update $T_{3^3,5}$, $T_{5^3,8}$, $T_{3^7,1}$ and $T_{13^3,2^4}$ in Boyd's table of type numbers in~\cite[p.~152]{B94}, in which 
	\begin{equation*}
		T_{3^3,5}=5,\,\,\,T_{5^3,8}=25,\,\,\,T_{3^7,1}=88,\,\,\,T_{13^3,2^4}=1040.
	\end{equation*}
	Re-evaluation via SageMath using her formula gives values identical to ours, indicating that the errors are computational in origin.
	\end{remark}
	\newpage
	
	\section{Computing local representation densities}\label{App:B}
	
	This appendix computes the local representation densities needed for the type number formula. While some results can be derived from~\cite[Theorem 3.1]{Y98}, we provide direct calculations for completeness.

\subsection{Elementary congruence counts}

\begin{lemma}\label{lemmafordensitiesatp1}
Let $p$ be odd, $p\nmid c$, and $t\ge1$. Then
\[
\#\{x\bmod p^t:x^2\equiv c\pmod{p^t}\}
 =1+\left(\frac cp\right).
\]
\end{lemma}
\begin{proof}
The assertion at $t=1$ is the definition of the Legendre symbol.
Each nonzero root has nonzero derivative and lifts uniquely.
\end{proof}

\begin{lemma}\label{lemmafordensitiesatp2}
If $p$ is odd and $p\nmid a$, then
\[
\sum_{y\bmod p}\left(\frac{a+y^2}{p}\right)=-1.
\]
\end{lemma}
\begin{proof}
Counting $x^2-y^2=a$ gives $p-1$ solutions, since
$(x-y)(x+y)=a$ and $2$ is invertible modulo $p$.
Counting first over $y$ gives $p+\sum_y(\frac{a+y^2}p)$.
\end{proof}

\begin{lemma}\label{lemmafordensitiesat2}
Let $c$ be odd and $t\ge3$. Then
\[
\#\{x\bmod2^t:x^2\equiv c\pmod{2^t}\}
 =\begin{cases}4,&c\equiv1\pmod8,\\0,&c\not\equiv1\pmod8.
   \end{cases}
\]
\end{lemma}
\begin{proof}
The square map on $(\mathbb Z/2^t\mathbb Z)^\times$ has kernel
of order $4$ and image the subgroup of units congruent to $1$
modulo $8$.
\end{proof}

We will also use two binary counts. If $c$ is a $p$-adic unit,
then
\[
\#\{(y,z)\bmod p^t:yz\equiv c\pmod{p^t}\}
 =(p-1)p^{t-1}.
\]
For the norm form of the unramified quadratic extension, the
corresponding number is $(p+1)p^{t-1}$.
For odd $p$ this form can be written $y^2-\epsilon z^2$;
at $2$ it can be written $y^2+yz+z^2$.
The norm count follows from the surjective norm map on units
modulo $p^t$.

\subsection{Odd primes}

\begin{proof}[Proof of Proposition~\ref{anisotropic}]
Put
\[
A_u(n)=d_{-\epsilon x^2+p^{2u+1}(y^2-\epsilon z^2),p}(n).
\]
For $u\ge1$, substituting $x=px'$ and dividing by $p^2$ gives
\[
A_u(p^2n)=pA_{u-1}(n).
\]
The factor $p$ is the change in the normalized congruence count.
For $p\nmid m$, direct counting gives
\[
A_0(m)=1-\left(\frac{-m}p\right),\qquad
A_0(pm)=1+p^{-1}.
\]
For the second equality, $p\mid x$ is forced and, after division
by $p$, the binary unramified norm must represent a unit.
Furthermore, if the value of the principal form is divisible
by $p^2$, then all three coordinates are divisible by $p$.
Indeed $p\mid x$, and the anisotropic binary norm modulo $p$
then forces $p\mid y,z$. Therefore
\[
A_0(p^2n)=p^{-1}A_0(n).
\]
If $v_p(n)<2u+1$, an odd valuation is impossible, whereas
$n=p^{2k}m$ gives
$A_u(n)=p^k(1-(\frac{-m}p))$.
If $v_p(n)\ge2u+1$, reduce $u$ times to $A_0$ and apply the
last recurrence. These give all four cases of the proposition.
\end{proof}

\begin{proof}[Proof of Proposition~\ref{isotropic}]
Write $S_v(n)=d_{-x^2-p^vyz,p}(n)$.
For $v\ge2$,
\[
S_v(p^2n)=pS_{v-2}(n).
\]
The two base forms satisfy, for $p\nmid m$,
\begin{align*}
S_0(m)&=1+p^{-1}\left(\frac{-m}p\right),&
S_0(pm)&=1-p^{-2},\\
S_1(m)&=1+\left(\frac{-m}p\right),&
S_1(pm)&=1-p^{-1}.
\end{align*}
Their recurrences are
\[
S_0(p^2n)=1-p^{-2}+p^{-1}S_0(n),\qquad
S_1(p^2n)=2(1-p^{-1})+p^{-1}S_1(n).
\]
For the first recurrence, the $p^2-1$ nonzero isotropic
vectors modulo $p$ are nonsingular. They contribute
$1-p^{-2}$; vectors divisible by $p$ contribute
$p^{-1}S_0(n)$.
For the second recurrence, the primitive solutions have
$p\mid x$ and exactly one of $y,z$ a unit. Each of these two
possibilities contributes $1-p^{-1}$.
The remaining solutions have all three coordinates divisible
by $p$ and contribute $p^{-1}S_1(n)$.

For $v_p(n)<v$, parity of the valuation of $x^2$ gives the
vanishing case and the unit-square-root case in the statement.
For $v=2s$ and $v_p(n)\ge v$, reduce to
$p^sS_0(n/p^{2s})$.
For $v=2s+1$ and $v_p(n)\ge v$, reduce to
$p^sS_1(n/p^{2s})$.
Iteration of the displayed base recurrences gives the stated
formulas for both parities of $v$.
\end{proof}

\subsection{Base densities at $2$}

\begin{lemma}[\cite{BJa12,B23}]\label{lemma:base_densities_2}
Put
\[
E(n)=d_{-x^2-yz,2}(n),\qquad
F(n)=d_{-x^2-2yz,2}(n).
\]
If $n=4^am$ with $a\ge0$ and $4\nmid m$, then
\[
E(n)=
\begin{cases}
\frac32,&m\equiv7\pmod8,\\
\frac32-2^{-a-1},&m\equiv3\pmod8,\\
\frac32-3\cdot2^{-a-2},&m\equiv1,2\pmod4,
\end{cases}
\]
and
\[
F(n)=
\begin{cases}
2,&m\equiv7\pmod8,\\
2-2^{-a},&m\equiv3\pmod8,\\
2-3\cdot2^{-a-1},&m\equiv1,2\pmod4.
\end{cases}
\]
\end{lemma}

\begin{lemma}\label{lemma:scaling_2}
For $g\ge2$ and every positive integer $n$,
\[
d_{-x^2-2^gyz,2}(4n)
 =2d_{-x^2-2^{g-2}yz,2}(n).
\]
For $u\ge1$,
\[
d_{3x^2-2^{2u+3}(y^2+yz+z^2),2}(4n)
 =2d_{3x^2-2^{2u+1}(y^2+yz+z^2),2}(n).
\]
\end{lemma}
\begin{proof}
The congruence forces $x$ even. Substitute $x=2x'$ and divide
by $4$. The proportion of allowable $x$ changes by $1/2$, while
the equation loses two powers of $2$, giving the factor $2$.
\end{proof}

\subsection{The anisotropic density at $2$}

\begin{proof}[Proof of Proposition~\ref{anisotropic2}]
Let $A_u(D)$ denote the density in that proposition.
For the principal form $A_0$, the square-root count gives
\[
A_0(m)=2\left(1-\left(\frac{-m}{2}\right)\right)
\quad(m\equiv3\pmod4).
\]
Indeed, for each $y,z$ the equation for $x^2$ has four roots
if $m\equiv3\pmod8$ and no roots if $m\equiv7\pmod8$.
If $m=4m_0$ with $m_0\equiv1,2\pmod4$, substitution $x=2x'$
gives $A_0(m)=2d_{3x^2-2(y^2+yz+z^2),2}(m_0)=3$.
Here exactly one parity of $x$ is possible, and the resulting
binary norm value is always odd, giving density $3/2$ before
the factor $2$.

For $D\equiv0,3\pmod4$, reduction modulo $16$ shows that every
representation of $4D$ by the principal form has all coordinates
even. Hence
\[
A_0(4D)=\frac12A_0(D).
\]
If $k\ge u$, Lemma~\ref{lemma:scaling_2} reduces the density to
$2^uA_0(4^{k-u}m)$, giving the two final cases.
If $k<u$, it reduces to $2^kA_{u-k}(m)$.
For odd $m$ the binary term is divisible by $8$ and the
square-root count gives $2^{k+1}(1-(\frac{-m}{2}))$.
For $m=4m_0$, the valuation and congruence modulo $4$ after
division by $4$ exclude a representation.
\end{proof}

\subsection{The isotropic density at $2$}

\begin{proof}[Proof of Proposition~\ref{isotropic2}]
Put $S_v(D)=d_{-x^2-2^{v+2}yz,2}(D)$.
The scaling lemma gives
\[
S_v(4D)=2S_{v-2}(D)\quad(v\ge2),\qquad
S_0(4D)=2E(D),\qquad S_1(4D)=2F(D).
\]
For odd $m\equiv3\pmod4$, the base values are
\[
S_0(m)=\begin{cases}3,&m\equiv7\pmod8,\\1,&m\equiv3\pmod8,
\end{cases}
\qquad
S_1(m)=2\left(1+\left(\frac{-m}{2}\right)\right).
\]
The first count follows from the binary product density as in
Lemma~\ref{lemma:base_densities_2}; the second follows directly
from the criterion for an odd unit to be a square in $\mathbb Z_2$.

If $v=2s$, repeated scaling gives the unit-square-root case
when $k<s$, and reduces the remaining cases to $S_0$ and $E$.
In particular, for $m=4m_0$ and $k\ge s$,
\[
S_{2s}(4^km)
 =2^{s+1}E(4^{k-s}m_0)
 =3\left(2^s-2^{2s-k-1}\right).
\]
If $v=2s+1$, the odd-$m$ case reduces to the unit-square-root
count for $k\le s$, and to $F$ for $k\ge s+1$.
For $m=4m_0$, the reduction to $F$ already applies at $k=s$:
\[
S_{2s+1}(4^km)
 =2^{s+1}F(4^{k-s}m_0)
 =2^{s+2}-3\cdot2^{2s-k}\qquad(k\ge s).
\]
For $k<s$ the valuation and unit-square conditions exclude a
representation. This explains both the exponent and the distinct
threshold in the even-$m$ branch of the odd-level formula.
The other branches follow by the displayed base values and
Lemma~\ref{lemma:base_densities_2}.
\end{proof}

	\bibliographystyle{plain}
	\bibliography{TQF}
	
\end{document}